\documentclass[reqno]{amsart}
\usepackage[left=1in,right=1in,top=1in,bottom=1in]{geometry}
\usepackage{tikz}
\usetikzlibrary{decorations.pathreplacing,shapes.geometric}
\usetikzlibrary{calc,positioning}
\usepackage{amsthm}
\usepackage{amscd}
\usepackage{amsfonts}
\usepackage{amsmath}
\usepackage{amssymb}
\usepackage{mathrsfs}
\usepackage{multirow}
\usepackage{verbatim}
\usepackage{url}
\usepackage{color}
\usepackage{graphicx}
\usepackage{yfonts}
\tikzset{Box/.style={very thick, rounded corners}}
\tikzset{marked/.style={star, star point height = .75mm, star points =5, fill=black,minimum size=2mm, inner sep=0mm} }
\tikzset{verythickline/.style = {line width=7pt}}
\tikzset{thickline/.style = {line width=5pt}}
\tikzset{medthick/.style = {line width=3pt}}
\tikzset{med/.style = {line width=2pt}}
\tikzset{count/.style = {fill=white,circle,draw,thin, inner sep=2pt}}
\tikzset{rcount/.style = {fill=white,rectangle,draw,thin,inner sep=2pt, rounded corners}}
\tikzset{cpr/.style = {draw,fill=white,rectangle,thin, rounded corners}}

\begin{document}

%%%%%%%%%%%%%%%%%%%%%%%%
% Custom Commands %
%%%%%%%%%%%%%%%%%%%%%%%%

%\renewcommand{\theenumii}{\roman{enumii}}
%\renewcommand{\labelenumii}{(\theenumii)}
\newcommand{\supp}{\text{supp}}
\newcommand{\Aut}{\text{Aut}}
\newcommand{\Gal}{\text{Gal}}
\newcommand{\Inn}{\text{Inn}}
\newcommand{\Irr}{\text{Irr}}
\newcommand{\Ker}{\text{Ker}}
\newcommand{\N}{\mathbb{N}}
\newcommand{\Z}{\mathbb{Z}}
\newcommand{\Q}{\mathbb{Q}}
\newcommand{\R}{\mathbb{R}}
\newcommand{\C}{\mathbb{C}}
\renewcommand{\H}{\mathcal{H}}
\newcommand{\B}{\mathcal{B}}
\newcommand{\A}{\mathcal{A}}
\newcommand{\K}{\mathcal{K}}
\newcommand{\M}{\mathcal{M}}

\newcommand{\J}{\mathscr{J}}
\newcommand{\D}{\mathscr{D}}
\renewcommand{\P}{\mathscr{P}}

\newcommand{\ul}[1]{\underline{#1}}

\newcommand{\I}{\text{I}}
\newcommand{\II}{\text{II}}
\newcommand{\III}{\text{III}}

\newcommand{\<}{\left\langle}
\renewcommand{\>}{\right\rangle}
\renewcommand{\Re}[1]{\text{Re}\ #1}
\renewcommand{\Im}[1]{\text{Im}\ #1}
\newcommand{\dom}[1]{\text{dom}\,#1}
\renewcommand{\i}{\text{i}}
\renewcommand{\mod}[1]{(\operatorname{mod}#1)}
\newcommand{\mb}[1]{\mathbb{#1}}
\newcommand{\mc}[1]{\mathcal{#1}}
\newcommand{\mf}[1]{\mathfrak{#1}}
\newcommand{\im}{\operatorname{im}}
%\renewcommand{\span}{\text{span}}

%%%%%%%%%%%%%%%%%%%%%%%%
% Theorem Environments %
%%%%%%%%%%%%%%%%%%%%%%%%

\newtheorem{thm}{Theorem}[section]
\newtheorem{prop}[thm]{Proposition}
\newtheorem{lem}[thm]{Lemma}
\newtheorem{cor}[thm]{Corollary}
\newtheorem{innercthm}{Theorem}
\newenvironment{cthm}[1]
{\renewcommand\theinnercthm{#1}\innercthm}
{\endinnercthm}
\newtheorem{innerclem}{Lemma}
\newenvironment{clem}[1]
{\renewcommand\theinnerclem{#1}\innerclem}
{\endinnerclem}

\theoremstyle{definition}
\newtheorem{defi}[thm]{Definition}
\newtheorem{ex}[thm]{Example}
\newtheorem*{exs}{Examples}
\newtheorem{rem}[thm]{Remark}
\newtheorem{innercdefi}{Definition}
\newenvironment{cdefi}[1]
{\renewcommand\theinnercdefi{#1}\innercdefi}
{\endinnercdefi}

%%%%%%%%%%%%%%%%%%%%%%%%
% Title Page %
%%%%%%%%%%%%%%%%%%%%%%%%

\title{Free transport for finite depth subfactor planar algebras}

\author{Brent Nelson}
\address{UCLA Mathematics Department}
\email{bnelson6@math.ucla.edu}
\thanks{Research supported by NSF grants DMS-1161411 and DMS-0838680}

\begin{abstract}
Given a finite depth subfactor planar algebra $\mc{P}$ endowed with the graded $*$-algebra structures $\{Gr_k^+ \mc{P}\}_{k\in\N}$ of Guionnet, Jones, and Shlyakhtenko, there is a sequence of canonical traces $Tr_{k,+}$ on $Gr_k^+\mc{P}$ induced by the Temperley-Lieb diagrams and a sequence of trace-preserving embeddings into the bounded operators on a Hilbert space. Via these embeddings the $*$-algebras $\{Gr_k^+\mc{P}\}_{k\in \N}$ generate a tower of non-commutative probability spaces $\{M_{k,+}\}_{k\in\N}$ whose inclusions recover $\mc{P}$ as its standard invariant. We show that traces $Tr_{k,+}^{(v)}$ induced by certain small perturbations of the Temperley-Lieb diagrams yield trace-preserving embeddings of $Gr_k^+\mc{P}$ that generate the same tower $\{M_{k,+}\}_{k\in\N}$.
\end{abstract}

\maketitle

%%%%%%%%%%%%%%%%%%%%%%%%
% Introduction %
%%%%%%%%%%%%%%%%%%%%%%%%

\section{Introduction}

Despite the relatively innocuous definition of a subfactor, Jones showed in \cite{J83}, \cite{J99}, and \cite{J00} that there is in fact an incredibly rich structure underlying the inclusion of one $\text{II}_1$ factor in another. In particular, one can associate to a subfactor $N\subset M$ its standard invariant: a planar algebra. It was later shown by Popa in \cite{P95} that in fact every subfactor planar algebra can be realized through this association.

In \cite{GJS10} Guionnet, Jones, and Shlyakhtenko produce an alternate proof of this fact by constructing the subfactors via free probabilistic methods. Given a subfactor planar algebra $\mc{P}$, for each $k\geq 0$ one can turn $Gr_k^+\mc{P}=\oplus_{n\geq k} \mc{P}_{n,+}$ into a $*$-algebra with a trace $Tr_{k,+}$ defined by a particular pairing with Temperley-Lieb diagrams.  Then each $Gr_k^+\mc{P}$ embeds into the bounded operators on a Hilbert space and generates a $\text{II}_1$ factor $M_{k,+}$. Moreover, one can define inclusion maps $i^{k-1}_k\colon M_{k-1,+}\to M_{k,+}$ so that the standard invariant associated to the subfactor inclusion $i_k^{k-1}(M_{k-1,+})\subset M_{k,+}$ (for any $k\geq 1$) recovers $\mc{P}$ as its standard invariant. The embedding relies on the fact that a subfactor planar algebra $\mc{P}$ always embeds into the planar algebra of a bipartite graph $\mc{P}^\Gamma$ (\emph{cf.} \cite{J00}, \cite{JP10}, and \cite{MW10}). 

It turns out that $Gr_0^+\mc{P}$ embeds as a subalgebra of a free Araki-Woods factor. Free Araki-Woods factors and their associated free quasi-free states, studied by Shlyakhtenko in \cite{S97}, are type $\mathrm{III}_\lambda$ factors, $0<\lambda\leq 1$, and can be thought of as the non-tracial analogues of the free group factors. They are constructed starting from a strongly continuous one-parameter group of orthogonal transformations $\{U_t\}_{t\in\R}$ on a real Hilbert space $\H_\R$. When $U_t=1$ for all $t$, this construction simply yields the free group factor $L(\mathbb{F}_{\dim{\H_\R}})$. Stone's theorem guarantees the existence of a positive, non-singular generator $A$ satisfying $A^{it}=U_t$ for all $t\in \R$. It was shown in \cite{S97} that the type classification of a free Araki-Woods factor is determined by the spectrum of the generator $A$. Moreover, the action of the modular automorphism group is well known and also depends explicitly on $A$.  In \cite{N13}, by adapting the free transport methods of Guionnet and Shlyakhtenko (\emph{cf.} \cite{GS11}), it was shown that non-commutative random variables whose joint law is ``close'' to a free quasi-free state in fact generate a free Araki-Woods factor. In particular, the finitely generated $q$-deformed Araki-Woods algebras were shown to be isomorphic to the free Araki-Woods factor for small $|q|$. In this paper we show that the free transport machinery can be encoded via planar tangles and provide an application of free transport to finite depth subfactor planar algebras.

Let $\mc{P}$ be a finite depth subfactor planar algebra and $Tr\colon \mc{P}\to\C$ be the state induced by the Temperley-Lieb diagrams via duality. By using the transport construction methods of \cite{N13}, we show that we can perturb the embedding constructed in \cite{GJS10} to make it state-preserving for states on $\mc{P}$ which are ``close'' to $Tr$. Moreover, the von Neumann algebra generated by the subfactor planar algebra via this embedding is unchanged. In this context, if $\mc{P}$ embeds into $\mc{P}^\Gamma$ and $\mu$ is the Perron-Frobenius eigenvector for the bipartite graph $\Gamma$, then the generator $A$ associated to the free Araki-Woods algebra will be determined by $\mu$.

The free transport methods in \cite{GS11} and \cite{N13} apply only to joint laws of finitely many non-commutative random variables. Since each edge in the graph $\Gamma$ will correspond to a non-commutative random variable, we can only consider finite depth subfactor planar algebras with these methods.

\subsection*{Acknowledgments}

I would like to thank Arnaud Brothier and Michael Hartglass for many useful discussions and Dimitri Shlyakhtenko for the initial idea of the paper, many helpful suggestions, and his general guidance.

%%%%%%%%%%%%%%%%%%%%%%%%
% Planar Algebras %
%%%%%%%%%%%%%%%%%%%%%%%%

\section{Planar Algebras}\label{Planar_algebras}

We briefly recall the definitions of a planar algebra and planar tangle. For additional details, see \cite{J99}, \cite{GJS10}, and \cite{GJSZJ12}, \cite{HP14}.

\begin{defi}
A \emph{planar algebra} is a collection of graded vector spaces $\mc{P}=\{\mc{P}_{n,\epsilon}\}_{n\geq 0,\epsilon\in\{\pm\}}$ possessing a conjugate linear involution $*$. For each $k\geq 0$ we call $\mc{P}_k:=\mc{P}_{k,+}\oplus \mc{P}_{k,-}$ the \emph{$k$-box space of $\mc{P}$}. A planar algebra also admits an action by \emph{planar tangles}. A planar tangle consists of an output disc $D_0\subset \R^2$ and several input discs $D_1,\ldots, D_r\subset D_0$, each disc $D_j$, $0\leq j\leq r$, having $2k_j$ boundary points ($k_j\geq 0$). These boundary points divide the boundaries of the discs into separate intervals and the \emph{distinguished interval} is marked with a ``$\star$.'' Each boundary point is paired with another boundary point (potentially from a distinct disc) and connected via non-crossing strings in $D_0\setminus (D_1\cup\cdots\cup D_r)$. The strings divide $D_0\setminus (D_1\cup\cdots\cup D_r)$ into several regions which are then shaded black or white so that adjacent regions have different shades.

Let $T$ be a planar tangle whose output disc $D_0$ has $2k_0$ boundary points and whose input discs $D_1,\ldots, D_r$ have $2k_1,\ldots, 2k_r$ boundary points. For each $j=0,\ldots, r$ we define $\epsilon_j\in\{+,-\}$ to be $+$ if the distinguished interval of $D_j$ borders a white region and $-$ otherwise. Then $T$ corresponds to a multilinear map $Z_T\colon \mc{P}_{k_1,\epsilon_1}\times\cdots\times \mc{P}_{k_r,\epsilon_r}\to \mc{P}_{k_0,\epsilon_0}$. These maps satisfy the following conditions.
\begin{enumerate}
\item \textbf{Isotopy invariance:} if $F$ is an orientation preserving diffeomorphism of $\R^2$ then $Z_{T}=Z_{F(T)}$.
\item \textbf{Naturality:} gluing planar tangles into one another corresponds to composing the multilinear maps.
\item \textbf{Involutive:} if $G$ is an orientation reversing diffeomorphism of $\R^2$ then
\[
Z_T(x_1,\ldots, x_r)^*=Z_{G(T)}(x_1^*,\ldots, x_r^*).
\]
\end{enumerate}
Furthermore, there is a canonical scalar $\delta$ associated with $\mc{P}$ with the property that a tangle with a closed loop is equivalent to $\delta$ times the tangle with the closed loop removed.
.
\end{defi}

In light of the isotopy invariance of the planar tangles, we will usually depict the input discs as rectangles with all strings emanating from the top side and the distinguished interval being formed by the other sides. For example:
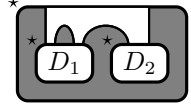
\[
\begin{tikzpicture}[thick, scale=.5]
\draw[Box, fill=gray] (-0.5,0.5) rectangle (4,3);
\node[marked, scale=.8, above left] at (-0.5,3) {};
\draw[fill=white] (0.25,3) -- (3,3) --++ (0,-1) --++ (-.55,0) arc(0:180:0.6 and 0.5) --++ (-.25,0) arc (0:180:0.25 and 0.5) --++ (-0.25,0) --++ (0,1);
\draw[Box, fill=white] (0,1) rectangle (1.5,2);
\node at (0.75,1.5) {$D_1$};
\node[marked, scale=.8, above left] at (0,2) {};
\draw[thick,] (0.25,2) --++ (0,1);
\draw[thick,] (0.5,2) arc (180:0:0.25 and 0.5);
\draw[thick] (1.25,2) arc (180:0:0.6 and 0.5);
\draw[Box, fill=white] (2,1) rectangle (3.5,2);
\node at (2.75,1.5) {$D_2$};
\node[marked, scale=.8, above left] at (2,2) {};
\draw[thick] (3,2) --++ (0,1);
\end{tikzpicture}
\]
corresponds to a multilinear map $\mc{P}_{2,-}\times \mc{P}_{1,-}\to \mc{P}_{1,-}$. We shall usually omit drawing the output disc and the shading.

Given a planar algebra $\mc{P}$ we define $Gr_k^{\pm}\mc{P}=\bigoplus_{n\geq k} \mc{P}_{n,\pm}$ and $Gr_k\mc{P} = Gr_k^+\mc{P}\oplus Gr_k^{-}\mc{P}$ for each $k\geq 0$. An element of $x\in Gr_k\mc{P}$ can be visually represented as
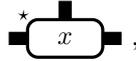
\[
\begin{tikzpicture}[thick, scale=.5]
\draw[Box] (0, 0) rectangle (2,1);
\node at (1,0.5) {$x$};
\node[marked, scale=.8, above left] at (0,1) {};
\draw[thickline] (0, 0.5) --++ (-0.5,0);
\draw[thickline] (2, 0.5) --++ (0.5,0);
\draw[thickline] (1,1) --++ (0,0.5);
\node[right] at (2.5,0.3) {,};
\end{tikzpicture}
\]
where the thick lines on the left and right each represent $k$ strings, the thick line on top is an even number of strings (possibly zero), and the shading of the region bordered by the distinguished interval varies according to the components of $x$. $Gr_k\mc{P}$ is endowed with the multiplication
\[
\begin{tikzpicture}[thick, scale=.5]
\node[left] at (-0.5,0.5) {$x\wedge_k y=$};
\draw[Box] (0, 0) rectangle (2,1);
\node at (1,0.5) {$x$};
\node[marked, scale=.8, above left] at (0,1) {};
\draw[thickline] (0, 0.5) --++ (-0.5,0);
\draw[thickline] (2, 0.5) --++ (0.5,0);
\draw[thickline] (1,1) --++ (0,0.5);
\draw[Box] (3, 0) rectangle (5,1);
\node at (4,0.5) {$y$};
\node[marked, scale=.8, above left] at (3,1) {};
\draw[thickline] (3, 0.5) --++ (-0.5,0);
\draw[thickline] (5, 0.5) --++ (0.5,0);
\draw[thickline] (4,1) --++ (0,0.5);
%\node[right] at (5.5,0.3) {,};

\end{tikzpicture}
\]
(with products of components with incompatible shadings taken to be zero), and the involution
\[
\begin{tikzpicture}[thick, scale=.5]
\node[left] at (-0.5,0.75) {$x^\dagger=$};
\draw[Box] (-0.5,-0.25) rectangle (2.5,1.5);
\node[marked, scale=.8, above left] at (-0.5,1.5) {};
\draw[Box] (0, 0) rectangle (2,1);
\node at (1,0.5) {$x^*$};
\node[marked, scale=.8, above right] at (2,1) {};
\draw[thickline] (0, 0.5) --++ (-0.5,0);
\draw[thickline] (2, 0.5) --++ (0.5,0);
\draw[thickline] (1,1) --++ (0,0.5);
\node[right] at (2.5,0.3) {.};
\end{tikzpicture}
\]

Now let $\text{TL}\subset \mc{P}$ be the canonical copy of the Temperley-Lieb planar algebra, and $TL_n$ the sum of all the Temperley-Lieb diagrams with $2n$ boundary points (including both shadings). Then we consider the $\mc{P}_{0,+}\oplus\mc{P}_{0,-}$ valued map $Tr_k$ on $Gr_k\mc{P}$ defined for $x\in \mc{P}_{n+k,+}\oplus \mc{P}_{n+k,-}$ by
\[
\begin{tikzpicture}[thick, scale=.5]
\node[left] at (4,0.5) {$\displaystyle Tr_k(x)=\frac{1}{\delta^k}$};
\draw[Box] (4.5,0) rectangle (6.5,1);
\node at (5.5,0.5) {$x$};
\node[marked, scale=.8, above left] at (4.5,1) {};
\draw[thickline] (5.5,1) --++ (0,0.5);
\draw[thickline] (4.5,0.5) arc(90:270:0.5)--++ (2,0) arc(-90:90:0.5);
\draw[Box] (4.5,1.5) rectangle (6.5,2.5);
\node at (5.5, 2) {$TL_n$};
\node[marked, scale=.8, below right] at (6.5,1.5) {};
\node[below right] at (7.0,0.5) {.};
\end{tikzpicture}
\]

Let $Gr_0[[\mc{P}]]$ denote the family of formal power series on elements in $Gr_0\mc{P}$. As a vector space, this is equivalent to $\displaystyle \prod_{\epsilon\in\{\pm\}, n\geq 0} \mc{P}_{n,\epsilon}$. Then if $TL_\infty:=\sum_{n\geq 0} TL_n \in Gr_0[[\mc{P}]]$, we can define $Tr_k(x)$ for a general $x\in Gr_k\mc{P}$ simply by
\[
\begin{tikzpicture}[thick, scale=.5]
\node[left] at (4,0.5) {$\displaystyle Tr_k(x)=\frac{1}{\delta^k}$};
\draw[Box] (4.5,0) rectangle (6.5,1);
\node at (5.5,0.5) {$x$};
\node[marked, scale=.8, above left] at (4.5,1) {};
\draw[thickline] (5.5,1) --++ (0,0.5);
\draw[thickline] (4.5,0.5) arc(90:270:0.5)--++ (2,0) arc(-90:90:0.5);
\draw[Box] (4.5,1.5) rectangle (6.5,2.5);
\node at (5.5, 2) {$TL_\infty$};
\node[marked, scale=.8, below right] at (6.5,1.5) {};
\node[below right] at (7.0,0.5) {,};
\end{tikzpicture}
\]
since the only components of $TL_\infty$ which will contribute non-zero terms are those matching the components of $x$, of which there are a finite number. In fact, given any $f\in Gr_0[[\mc{P}]]$ we can define a $\mc{P}_{0,+}\oplus\mc{P}_{0,-}$ valued map with
\begin{equation}\label{duality}
\begin{tikzpicture}[thick, scale=.5]
\node[left] at (4,0.5) {$Gr_k\mc{P}\ni x\longmapsto$};
\draw[Box] (4.5,0) rectangle (6.5,1);
\node at (5.5,0.5) {$x$};
\node[marked, scale=.8, above left] at (4.5,1) {};
\draw[thickline] (5.5,1) --++ (0,0.5);
\draw[thickline] (4.5,0.5) arc(90:270:0.5)--++ (2,0) arc(-90:90:0.5);
\draw[Box] (4.5,1.5) rectangle (6.5,2.5);
\node at (5.5, 2) {$f$};
\node[marked, scale=.8, below right] at (6.5,1.5) {};
\node[below right] at (7.0,0.5) {.};
\end{tikzpicture}
\end{equation}

%%%%%%%%%%%%%%%%%%%%%%%%
% Subfactor planar algebras

\subsection{Subfactor planar algebras}

\begin{defi}
A \emph{subfactor planar algebra} $\mc{P}$ is a planar algebra satisfying:
\begin{enumerate}
\item $\dim(\mc{P}_{n,\pm})<\infty$ for all $(n,\pm)$;

\item $\dim(\mc{P}_{0,\pm})=1$;

\item for each $(n,\pm)$ the sesquiliner form where the thick string denotes $2n$ strings
\[
	\begin{tikzpicture}[thick, scale=.5]
	\node[left] at (0.5,1) {$\<b,a\>=$};

	\draw[Box] (0.5,0.5) rectangle (1.5,1.5);
	\node at (1,1) {$b^*$};
	\node[marked, scale=.8, above left] at (0.5,1.5) {};
	\draw[thickline] (1,1.5) arc(180:0:0.75);
	\draw[Box] (2,0.5) rectangle (3,1.5);
	\node at (2.5,1) {$a$};
	\node[marked, scale=.8, above left] at (2,1.5) {};
	\node[right] at (4,1) {$a,b\in\mc{P}_{n,\pm}$};
	\end{tikzpicture}
\]
(shaded according to $\pm$) is positive definite; and
\item the equality
\[
	\begin{tikzpicture}[thick, scale=.5]
	\draw[Box] (0,0) rectangle (1,1);
	\node at (0.5,0.5) {$x$};
	\node[marked, scale=.8, above left] at (0,1) {};
	\draw[thick] (0.25,1) arc(0:180:0.35 and 0.5) --++ (0,-.85) arc(180:270:0.45) --++ (1,0) arc(270:360:0.45) --++ (0,.85) arc (0:180:0.35 and 0.5);
	\node at (2.1,0.5) {$=$};
	\draw[Box] (2.7,0) rectangle (3.7,1);
	\node at (3.2,0.5) {$x$};
	\node[marked, scale=.8, above left] at (2.7,1) {};
	\draw[thick] (2.95,1) arc(180:0:0.25 and 0.5);
	\end{tikzpicture}
\]
holds for any $x\in \mc{P}_{1,\pm}$.
\end{enumerate}
\end{defi}

\begin{rem}
As in condition (3) above, all inner products in this paper will be complex linear in the second coordinate.
\end{rem}

The condition $\dim(P_{0,\pm})=1$ implies that each $P_{0,\pm}$ is isomorphic to $\C$ as $C^*$-algebras with the multiplication
\[
\begin{tikzpicture}[thick, scale=.5]
\node[left] at (0.5,1) {$ab=$};

\draw[Box] (0.5,0.5) rectangle (1.5,1.5);
\node at (1,1) {$a$};
\node[marked, scale=.8, above left] at (0.5,1.5) {};
\draw[Box] (2,0.5) rectangle (3,1.5);
\node at (2.5,1) {$b$};
\node[marked, scale=.8, above left] at (2,1.5) {};
\node[below right] at (3,1) {.};
\end{tikzpicture}
\]

Because of property $(2)$, the maps $Tr_k$ defined above are in fact $\C^2$-valued and we think of them as scalar valued when restricted to either $Gr_k^+\mc{P}$ or $Gr_k^-\mc{P}$. Write $Tr_k(x)=(Tr_{k,+}(x), Tr_{k,-}(x))$ for the two components, and let $TL_\infty^+$ (resp. $TL_\infty^-$) be the formal sum of all Temperley-Lieb diagrams whose distinguished interval borders an unshaded (resp. shaded) region. Then $Tr_{k,\pm}$ are equivalently defined using the same tangle as $Tr_k$ but replacing $TL_{\infty}$ with $TL_{\infty}^+$ or $TL_\infty^-$.

We extend the inner product from property (3) to all of $Gr_0\mc{P}$ with the convention that $\mc{P}_{n,\epsilon}$ is orthogonal to $\mc{P}_{m,\mu}$ when $(n,\epsilon)\neq (m,\mu)\in \N\times \{+,-\}$. Then $Tr_{0,\pm}(x)=\<TL_\infty^{\pm}, x\>$. More generally, if $\phi_{\pm}\colon Gr_0^{\pm}\mc{P}\to \C$ are linear functionals and $\phi_0=\phi_+\oplus \phi_-$  then there exists an element $f\in Gr_0[[\mc{P}]]$ so that $\phi_0(x)=\<f^*, x\>$. Hence we can define $\phi_k \colon Gr_k\mc{P}\to \C^2$ for each $k$ via (\ref{duality}). We also note that if $\phi_0$ is positive then $f=f^*$.

%%%%%%%%%%%%%%%%%%%%%%%%
% Planar algebra of a bipartite graph

\subsection{Planar algebra of a bipartite graph}

For a more thorough treatment of the following section, please see Sections 2 and 4 of \cite{GJS10} (specifically subsections 2.4, 2.5, 4.1, 4.2, and 4.3).

Let $\Gamma=(V,E)$ be an oriented bipartite graph with positive vertices $V_+\subset V$ and negative vertices $V_-=V\setminus V_+$. Given an edge $e\in E$, we let $s(e),t(e)\in E$ denote its beginning and ending vertex, respectively, and let $e^\circ$ denote the edge with the opposite orientation (i.e. $s(e^\circ)=t(e)$ and $t(e^\circ)=s(e)$). Then $E_+=\{e\in E\colon s(e)\in V_+\}$ is the set of edges starting on a positive vertex, and $E_-=\{e\in E\colon s(e)\in V_-\}= \{e^\circ \colon e\in E_+\}$.

Let $L$ denote the set of loops in $\Gamma$ where a loop traveling along edges $e_1,e_2,\ldots, e_n$ (in that order) is written as $e_1e_2\cdots e_n$. Since $\Gamma$ is bipartite, any loop will consist of an even number of edges and so we let $L_n$ for $n\geq 0$ denote the loops of length $2n$ (with $L_0=V$). We further sort the loops according to whether they start with a positive or negative vertex and denote these by $L_{n,+}$ and $L_{n,-}$, respectively. Then for each $n\geq 0$, we consider the vector space $\mc{P}_{n,+}^\Gamma$ (resp. $\mc{P}_{n,-}^\Gamma$) of bounded functions on $L_{n,+}$ (resp. $L_{m,-}$).

When $|E|<\infty$ (and consequently $|L_{m,\pm}|<\infty$ for each $n$), the vector spaces $\mc{P}_{n,\pm}^\Gamma$ are finite dimensional and spanned by the delta functions supported on individual loops in $L_{n,\pm}$. Letting $u\in L_{n,\pm}$ serve as notation for both the loop and the delta function supported on said loop, we write
\begin{align*}
w=\sum_{u\in L_{n,\pm}} \beta_w(u) u
\end{align*}
for elements $w\in \mc{P}_{n,\pm}^\Gamma$, where $\beta_w(u)\in \C$.

We define the following involution on $\mc{P}_{n,\pm}^\Gamma$:
\begin{align*}
w^*:=\sum_{u\in L_{n,\pm}} \overline{\beta_w(u)} u^{op},
\end{align*}
where $u^{op}=e_n^\circ\cdots e_1^\circ$ when $u=e_1\cdots e_n$.

Let $A_\Gamma$ be the adjacency matrix for the graph $\Gamma$. By the Perron-Frobenius theorem, $A_\Gamma$ has a unique largest eigenvalue $\delta>0$ with eigenvector $\mu$ satisfying $\mu(v)>0$ for all $v\in V$. We note that the eigenvalue condition $A_\Gamma \mu=\delta \mu$ guarantees $\frac{\mu(v)}{\mu(w)}<\delta$ for all adjacent vertices $v,w\in V$.

The map $Z_T$ associated to a planar tangle is defined as follows. Replace $T$ with an isotopically equivalent tangle whose input and output discs are rectangles with boundary points along the top edges and distinguished interval forming the side and bottom edges. Assume that $D_0$ and $D_1,\ldots, D_r$ are the input and output discs, respectively, that $D_j$ has $2k_j$ boundary points, and that the distinguished interval of $D_j$ has the shading $\epsilon_j\in \{+,-\}$, $0\leq j\leq r$. Let $u_j\in L_{k_j,\epsilon_j}$ for each $j$, and assign each edge in $u_j$ to a boundary point on $D_j$. The edges are assigned in order with the leftmost boundary point corresponding to the first edge and the rightmost boundary point corresponding to the last edge. We set $Z_T(u_1,\ldots, u_r)\equiv 0$ unless every boundary point, say corresponding to an edge $e$, is connected to a boundary point of $D_0$ or is connected to a boundary point of another input disc corresponding to the edge $e^\circ$. When the latter holds, each string is labeled by a single edge (and its opposite) and consequently the regions in $D_0\setminus (D_1\cup\cdots\cup D_r\cup\{\text{strings}\})$ can be labeled by vertices: traversing the regions adjacent to $D_j$ clockwise corresponds to traveling along the vertices in the loop $u_j$. In this case, $Z_T(u_1,\ldots, u_r)$ is supported on the loop $f_1\cdots f_{2k_0}$, where $f_l=e$ if the $l$th boundary point of $D_0$ is connected to the boundary point of an input disc corresponding to the edge $e$. The value of this function is
	\begin{align*}
		[Z_T(u_1,\ldots, u_r)](f_1\cdots f_{2k_0}) =\delta^p \prod_{\gamma\in\{\text{strings in $T$}\}} \left(\frac{\mu(t(e_\gamma))}{\mu(s(e_\gamma))}\right)^{-\frac{\theta_\gamma}{2\pi}},
	\end{align*}
where $p$ is the number of closed loops in $T$, $e_\gamma$ is the edge corresponding to the boundary point at the start of the string $\gamma$, and $\theta_\gamma$ is the total winding angle of the string $\gamma$ (counter-clockwise being the direction of positive angles). We then multilinearly extend to $Z_T$ to $P_{k_1,\epsilon_1}\times\cdots\times P_{k_r,\epsilon_r}$.

When the output disc has zero boundary points there is one region of $D_0\setminus (D_1\cup\cdots\cup D_r\cup\{\text{strings}\})$ bordered by the boundary of of $D_0$. If the above procedure labels this region $v_0\in V$, then $Z_T(u_1,\ldots, u_r)$ is supported on $v_0$ with the same value as above.

We have the following fact originally due to Jones (\emph{cf.} \cite{J00}, \cite{JP10}, and \cite{MW10}):

\begin{prop}
Let $\mc{P}$ be a subfactor planar algebra. Then there exists a bipartite graph $\Gamma$ and a planar algebra embedding $i\colon \mc{P}\to \mc{P}^\Gamma$.
\end{prop}

A subfactor planar algebra is of \emph{finite depth} if its associated Bratteli diagram has finite width. From this Bratteli diagram one constructs the \emph{principal graph} for the subfactor planar algebra, which plays the role of $\Gamma$ in the above proposition (\emph{cf.} \cite{J00} and \cite{JP10}). In particular, if $\mc{P}$ is of finite depth then $\Gamma$ can be taken to be a finite graph.

For the remainder of the paper we fix a finite depth subfactor planar algebra $\mc{P}$, along with finite bipartite graph $\Gamma$ and inclusion $i\colon \mc{P}\to \mc{P}^\Gamma$. We will use the notations $\<b,a\>_{\mc{P}}$ or $\<b,a\>_{\mc{P}^\Gamma}$ to distinguish between the pairings
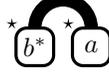
\[
	\begin{tikzpicture}[thick, scale=.5]
	\draw[Box] (0.5,0.5) rectangle (1.5,1.5);
	\node at (1,1) {$b^*$};
	\node[marked, scale=.8, above left] at (0.5,1.5) {};
	\draw[thickline] (1,1.5) arc(180:0:0.75);
	\draw[Box] (2,0.5) rectangle (3,1.5);
	\node at (2.5,1) {$a$};
	\node[marked, scale=.8, above left] at (2,1.5) {};
	\end{tikzpicture}
\]
occurring in $\mc{P}$ or $\mc{P}^\Gamma$.

Define the maps $\{Tr_k\}_{k\geq 0}$ for both $\mc{P}$ and $\mc{P}^\Gamma$ as above. As a planar algebra embedding, $i$ preserves the actions of tangles. Hence $Tr_k\circ i(x)= Tr_k(x)$ for all $x\in Gr_k\mc{P}$ and all $k\geq 0$. However, the $0$-box space of $\mc{P}^\Gamma$ is $\ell^\infty(V)$, so $Tr_k\circ i(x)$ is a function on $V$ satisfying
	\begin{equation}\label{trace_function_on_vertices}
		[Tr_k\circ i(x)](v) =\left\{\begin{array}{cc}
							Tr_{k,+}(x) & \text{if }v\in V_+\\
							Tr_{k,-}(x) & \text{if }v\in V_- \end{array}\right..
	\end{equation}
With this in mind we extend $i$ to an embedding $i\colon Gr_k\mc{P}\to Gr_k\mc{P}^\Gamma$. As the $*$-algebra structure of $Gr_k\mc{P}$ was defined using planar tangles, $i$ is a $*$-algebra embedding.

%%%%%%%%%%%%%%%%%%%%%%%%
% The GJS construction

\subsection{The Guionnet-Jones-Shlyakhtenko construction}\label{GJS_construction}
We let $H$ denote the complex Hilbert space with the edges $E$ of $\Gamma$ as an orthogonal basis and norms defined by
\begin{align*}
\|e\|^2=\left[\frac{\mu(s(e))}{\mu(t(e))}\right]^\frac{1}{2},
\end{align*}
and use the notation
\begin{equation*}
\sigma(e)=\left[\frac{\mu(t(e))}{\mu(s(e))}\right]^{\frac{1}{2}} = \|e\|^{-2}.
\end{equation*}
We define left and right actions of $\ell^\infty(V)$ on $H$ by
\begin{align*}
v\cdot e\cdot v'= \delta_{v=s(e)}\delta_{v'=t(e)} e,
\end{align*}
where $v$ denotes both the vertex and the delta function supported at that vertex. Thus $H$ is an $\ell^\infty(V)$-bimodule. We define an $\ell^\infty(V)$-valued inner product by
\begin{align*}
\<e,f\>_{\ell^\infty(V)} = \<e,f\>t(e) = \<e,f\>t(f).
\end{align*}

Let
	\begin{align*}
		\mc{F}_{\ell^\infty(V)}=\ell^\infty(V)\oplus \bigoplus_{n\geq 1} \H^{\otimes_{\ell^\infty(V)} n},
	\end{align*}
and observe that because the tensor product is relative to $\ell^\infty(V)$, non-zero elements $e_1\otimes\cdots\otimes e_n\in\mc{F}_{\ell^\infty(V)}$ correspond to paths $e_1\cdots e_n$ in $\Gamma$. Indeed:
\begin{align*}
e\otimes f= (e\cdot t(e))\otimes f = e\otimes (t(e)\cdot f) = \delta_{t(e)=s(f)} e\otimes f.
\end{align*}
For each $e\in E$ we define $\ell(e)\in \mc{B}(\mc{F}_{\ell^\infty(V)})$ by
\begin{align*}
\ell(e)& v = \delta_{t(e)=v} e\\
\ell(e)& e_1\otimes\cdots \otimes e_n= e\otimes e_1\otimes \cdots \otimes e_n,
\end{align*}
and then its adjoint is given by
\begin{align*}
\ell(e)^*& v= 0\\
\ell(e)^*& e_1\otimes \cdots \otimes e_n = \<e,e_1\>_{\ell^\infty(V)} e_2\otimes\cdots \otimes e_n.
\end{align*}
Notice that in the above formula $\<e,e_1\>_{\ell^\infty(V)} = \<e,e_1\> t(e_1)$ and that $t(e_1) e_2=e_2$ if this element is a path. The norm of this operator is given by
\begin{align*}
\| \ell(e)\| = \|\ell(e)^*\ell(e)\|^\frac{1}{2}= \|e\|.
\end{align*}

For each $e\in E$ we define the non-commutative random variable
\begin{align*}
c(e)=\ell(e)+\ell(e^\circ)^* \in \mc{B}(\mc{F}_{\ell^\infty(V)}),
\end{align*}
and consider the conditional expectation $\mc{E}\colon \mc{B}(\mc{F}_{\ell^\infty(V)})\to \ell^\infty(V)$ given by
\begin{align*}
\mc{E}(x) = \<1_{\ell^\infty(V)}, x1_{\ell^\infty(V)}\>_{\ell^\infty(V)},
\end{align*}
where $1_{\ell^\infty(V)}=\sum_{v\in V} v$ is the multiplicative identity in $\ell^\infty(V)$.

It is known that $(Gr_0^+ \mc{P}^\Gamma,Tr_0)$ embeds via
\begin{align*}
e_1\cdots e_{2n}\mapsto c(e_1)\cdots c(e_{2n})
\end{align*}
into the von Neumann algebra $(W^*(c(e)\colon e\in E_+), \mc{E})$ in a trace-preserving manner (\textit{cf.} Theorem 3 in \cite{GJS10}). In fact, all of $Gr_0\mc{P}^\Gamma$ embeds into $W^*(c(e)\colon e\in E)$ in a trace-preserving manner. Denote $\M:=W^*(c(e)\colon e\in E)$.

For each $v\in V$, we can define a state $\phi_v=\delta_v\circ \mc{E}$ and a weight
\begin{align*}
\phi=\sum_{v\in V} \phi_v.
\end{align*}
Then for $x\in Gr_0\mc{P}$, using (\ref{trace_function_on_vertices}) we see that
\begin{align*}
\phi\circ c\circ i(x) = |V_+|Tr_{0,+}(x) + |V_-|Tr_{0,-}(x) = \<|V_+| TL_\infty^+ + |V_-|TL_\infty^-, x\>.
\end{align*}
Consequently we define $\overline{TL}_\infty:=|V_+| TL_\infty^+ + |V_-|TL_\infty^-\in Gr_0[[\mc{P}]]$ and $\overline{Tr}_0(x)=\<\overline{TL}_\infty, x\>$ so that
\begin{align}\label{trace_of_subfactor_element}
\overline{Tr}_0(x)= \phi\circ c\circ i(x).
\end{align}
Consider the Fock space
\begin{align*}
\mc{F}=\C\Omega\oplus\bigoplus_{n\geq 1} \H^{\otimes n}
\end{align*}
(ignoring the $\ell^\infty(V)$-bimodule structure of $H$). Let $\varphi$ be the vacuum state on $\mc{B}(\mc{F})$. For each $e\in E$ we define $\hat{\ell}(e)\in \mc{B}(\mc{F})$ as above and let $\hat{c}(e)=\hat{\ell}(e)+\hat{\ell}(e^\circ)^*$. Extending $\hat{c}$ to loops by $e_1\cdots e_{2n}\mapsto \hat{c}(e_1)\cdots \hat{c}(e_{2n})$, it follows that $\phi\circ c=\varphi\circ \hat{c}$. Indeed, the GNS vector space associated to $\phi_v$ is isomorphic to the subspace of $\mc{F}$ spanned by elements of the form $e_1\otimes\cdots\otimes e_{2n}$ where $e_1\cdots e_{2n}\in L$ and $s(e_1)=t(e_{2n})=v$. Consequently,
\begin{align*}
\phi_v(c(e_1\cdots e_{2n}))=\varphi(\hat{c}(e_1\cdots e_{2n}).
\end{align*}
Since this holds for each $v$, $\phi(c(x))=\varphi(\hat{c}(x))$ by summing over the support of $x$ according to which vertex it starts at. Consequently, using (\ref{trace_of_subfactor_element}) we have
\begin{align}\label{trace_relation}
\overline{Tr}_0(x)= \varphi(\hat{c}\circ i(x))\qquad x\in Gr_0\mc{P}.
\end{align}
From now on, we will repress the embedding notation $i$ and consider $Gr_0\mc{P}$ as a subalgebra of $Gr_0\mc{P}^\Gamma$, although the traces of such elements will still be thought of as scalars so that (\ref{trace_relation}) makes sense.

We will use the notation $C_e=\hat{c}(e)$ for $e\in E$, and $M=W^*(C_e\colon e\in E)\subset \mc{B}(\mc{F})$. It turns out $M$ is a free Araki-Woods factor, which we demonstrate below, and thus this embedding lies in the scope of the transport results obtained in \cite{N13}.

%%%%%%%%%%%%%%%%%%%%%%%%
% Free Araki-Woods algebras %
%%%%%%%%%%%%%%%%%%%%%%%%

\section{Free Araki-Woods Algebras}

Each $C_e$ is a generalized circular element (\textit{cf.} \cite{S97}). Indeed, let $h=e/\|e\|$ and $g=e^\circ/\|e^\circ\|$ be normalized opposite edges. Then
\begin{align*}
C_e=\|e\| \hat{\ell}(h) + \|e\|^{-1} \hat{\ell}(g)^* = \|e\|( \hat{\ell}(h)+\sigma(e)\hat{\ell}(g)^*),
\end{align*}
so letting $\lambda(e)=\sigma(e)^2=\|e\|^{-4}$ we see that $C_e/\|e\|$ is a generalized circular element of precisely the form discussed in \cite{S97}. Consequently the $C_e$ will be linearly related to certain semicircular random variables, and the von Neumann algebra they generate will be a free Araki-Woods factor. We describe these semicircular elements presently. For $e\in E$ define
\begin{align*}
u(e)=\left\{\begin{array}{cl}	\frac{1}{\sqrt{\sigma(e)+\sigma(e^\circ)}} (e+e^\circ)	&	\text{if }e\in E_+\\
\frac{i}{\sqrt{\sigma(e)+\sigma(e^\circ)}} (e-e^\circ)	&	\text{if }e\in E_-\end{array}\right.,
\end{align*}
so that $u(e), u(e^\circ)$ are unit vectors. For each $e\in E$ let $X_e=\hat{\ell}(u(e))+\hat{\ell}(u(e))^*$, then it is easy to check that for $e\in E_+$
\begin{align}\label{linear_relation0}
	C_e &= \frac{\sqrt{\sigma(e)+\sigma(e^\circ)}}{2}\left(X_e-i X_{e^\circ}\right),\text{ and}\\
	C_{e^\circ} &= \frac{\sqrt{\sigma(e)+\sigma(e^\circ)}}{2}\left(X_e + i X_{e^\circ}\right).\nonumber
\end{align}
For each pair $e,f\in E$ let $\alpha_{ef}=\varphi(X_fX_e)=\<u(f),u(e)\>$. Then take $A\in M_{|E|}(\C)$ to be the matrix defined by $\left[\frac{2}{1+A}\right]_{ef}=\alpha_{ef}$. It follows that $A$ is a block-diagonal matrix in the sense that $[A]_{ef}=0$ unless $f\in \{e,e^\circ\}$. As this will be the case for many of the matrices considered in this paper, we adopt the following notation for $B\in M_{|E|}(\C)$ and $e\in E_+$:
\begin{align*}
B(e):=\left(\begin{array}{cc} [B]_{ee} & [B]_{ee^\circ}\\ {[B]_{e^\circ e}} & [B]_{e^\circ e^\circ}\end{array}\right) \in M_2(\C).
\end{align*}
In particular, we have
\begin{align*}
A(e)=
\left(
\begin{array}{cc}
\frac{1}{2}\left(\lambda(e)+\lambda(e)^{-1}\right)	&	-\frac{i}{2}\left(\lambda(e)-\lambda(e)^{-1}\right) \\
\frac{i}{2}\left(\lambda(e)-\lambda(e)^{-1}\right)	&	\frac{1}{2}\left(\lambda(e)+\lambda(e)^{-1}\right)
\end{array}
\right).
\end{align*}
Moreover, $A$ is positive with $\text{spectrum}(A)=\{\lambda(e)\}_{e\in E}$ and consequently,
\begin{align}\label{norm_of_A}
\|A\| = \max_{e\in E} \lambda(e) =\max_{e\in E} \frac{\mu(t(e))}{\mu(s(e))}<\delta.
\end{align}
Setting $U_t=A^{it}$ for $t\in \R$ gives a one-parameter orthogonal group with $[U_t]_{ef}=0$ when $f\not\in \{e,e^\circ\}$ and
\begin{align*}
U_t(e)=
\left(
\begin{array}{ll}
\cos(t\log{\lambda(e)})	&	-\sin(t\log{\lambda(e)})\\
\sin(t\log{\lambda(e)})	&	\cos(t\log{\lambda(e)})
\end{array}
\right)\qquad e\in E_+.
\end{align*}
It follows that $H$ is isomorphic to the closure of $\C^{|E|}$ with respect to the inner product
\begin{align*}
\<x,y\>_U=\<\frac{2}{1+A^{-1}}x,y\>\qquad x,y\in \C^{|E|},
\end{align*}
and this isomorphism is implemented by sending the standard basis of $\C^{|E|}$ to $\{u(e)\}_{e\in E}$ in the obvious way. Moreover, $M=W^*(C_e\colon e\in E)=W^*(X_e\colon e\in E)\cong \Gamma(\R^{|E|}, U_t)''$, where the latter von Neumann algebra is a free Araki-Woods factor.

%%%%%%%%%%%%%%%%%%%%%%%%
%	The differential operators

\subsection{The differential operators}

Since $M$ is a free Araki-Woods factor, all the machinery developed in \cite{N13} carries over and we proceed by translating it to the context of the generalized circular system $C=(C_e\colon e\in E)$. Let $X=(X_e\colon e\in E)$, then the linear relation in (\ref{linear_relation0}) can be stated succinctly as
\begin{align}\label{linear_relation}
C=UX,
\end{align}
where $U$ is the matrix with $[U]_{ef}=0$ for $f\not\in \{e,e^\circ\}$ and
\begin{align*}
U(e)=\frac{\sqrt{\sigma(e)+\sigma(e^\circ)}}{2} \left(\begin{array}{cc} 1 & -i\\ 1 & i\end{array}\right).
\end{align*}
Because of this linear relation, if we denote $\P=\C\<X_e\colon e\in E\>$ then these can be thought of as non-commutative polynomials in \emph{either} the $X_e$ or in the $C_e$. As elements of $\P$, the distinction is trivial; however, for the purposes of composition with elements of $\P^{|E|}$ it is necessary to indicate whether an element is being thought of as a function on the $C_e$ or the $X_e$.\par

Let $\{\delta_e\}_{e\in E}$ be the free difference quotients defined on $\P$ by $\delta_e(X_f)=\delta_{e=f}1\otimes 1$ and the Leibniz rule. We use the same conventions on $\P\otimes \P^{op}$ as those in \cite{N13}. The $\sigma$-difference quotients of \cite{N13} are given by
	\begin{align*}
		\partial_{u(e)}=\sum_{f\in E}\alpha_{fe}\delta_{f},
	\end{align*}
and these generate a new collection of derivations $\{\partial_e\}_{e\in E}$ via the linear relation in (\ref{linear_relation}):
	\begin{align*}
		\partial_e=[U]_{ee}\partial_{u(e)} + [U]_{e e^\circ}\partial_{u(e^\circ)}.
	\end{align*}
These can also be independently defined on $\P$ by $\partial_e(C_f)=\delta_{f=e^\circ}\sigma(e) 1\otimes 1$ and the Leibniz rule. We shall refer to the derivations $\{\partial_e\}_{e\in E}$ as \emph{c-difference quotients}.\par

For $Q\in \P^{|E|}$ we define $\J_c Q \in M_{|E|}(\P\otimes\P^{op})$ by $[\J_c Q]_{ef}=\partial_f Q_e$. In particular,
\begin{align}\label{derivative_of_C}
(\J_c C)(e) = \left(
\begin{array}{cc}
0	&	\sigma(e^\circ)1\otimes 1\\
\sigma(e)1\otimes 1	&	0
\end{array}
\right)\qquad e\in E_+.
\end{align}
Letting $\J_\sigma$ be the operator considered in \cite{N13} we have $[\J_\sigma Q]_{ef}=\partial_{u(f)}Q_e$ and
\begin{align}\label{linear_relation_derivative1}
\J_c Q=\J_\sigma Q\# U^{T}.
\end{align}
Using this and (\ref{linear_relation}) we see that
	\begin{align}\label{linear_relation2}
		\J_c C = U \# \J_\sigma X \# U^T = U\# \frac{2}{1+A} \# U^T,			
	\end{align}
and after noting that $\J_c C^{-1}= \J_c C$ we also have
	\begin{align}\label{linear_relation3}
		\frac{1+A}{2}=\J_\sigma X^{-1} = U^T\# \J_c C\# U.
	\end{align}

Let $\{\D_{u(e)}\}_{e\in E}$ be the $\sigma$-cyclic derivatives of \cite{N13}:
	\begin{align*}
		\D_{u(e)}(X_{e_1}\cdots X_{e_n})=\sum_{k=1}^n \alpha_{e e_k} \sigma_{-i}(X_{e_{k+1}}\cdots X_{e_n}) X_{e_1}\cdots X_{e_{k-1}},
	\end{align*}
and for $Q\in \P$ we let $\D Q$ be the $\sigma$-cyclic gradient of $Q$: $\D Q=(\D_{u(e)} Q\colon e\in E)$. We then define the \emph{c-cyclic derivatives} $\D_e = [U]_{ee}\D_{u(e)}+[U]_{ee^\circ}\D_{u(e^\circ)}$ for each $e\in E$. That is,
	\begin{align*}
		\D_e(C_{e_1}\cdots C_{e_n})&=\sigma(e^\circ)\sum_{k=1}^n \delta_{e_k=e^\circ} \sigma_{-i}^\varphi(C_{e_{k+1}}\cdots C_{e_n}) C_{e_1}\cdots C_{e_{k-1}}\\
			&=\sigma(e^\circ) \sum_{k=1}^n \delta_{e_k=e^\circ} \left(\prod_{l=k+1}^n \sigma(e_l)^2\right)C_{e_{k+1}}\cdots C_{e_n}C_{e_1}\cdots C_{e_{k-1}},
	\end{align*}
where we have used the action of the modular automorphism group $\varphi_t^\varphi$ on $C_e$ discussed in Lemma 5.(ii) of \cite{GJS10}. For $Q\in \P$ we define $\D_c Q=(\D_e Q\colon e\in E)$ as the \emph{c-cyclic gradient}. It then follows that
	\begin{align}\label{linear_relation_derivative2}
		\D_c Q=U \#\D Q.
	\end{align}
It is clear that the $c$-difference quotients and $c$-cyclic derivatives induce derivations on $Gr_0 \mc{P}^\Gamma$ through $\hat{c}$, and we denote these by $\partial_e$ and $\D_e$ as well. Suppose $eu$ is a loop (so that $u$ is a path from $t(e)$ to $s(e)$). Then $\partial_e u$ is zero unless $e^\circ$ is one of the edges traversed by $u$ in which case $\partial_e u$ is a tensor product $u_{\ell}\otimes u_r$ of two loops such that $u_{\ell}$ starts at $t(e)$ and $u_r$ starts at $s(e)$. If $u$ itself is a loop, then $\D_e u$ is zero unless $e^\circ$ is traversed by $u$ in which case $\D_e u$ is path starting at $s(e)$ and ending at $t(e)$.

We next encode the action of these differential operators on $Gr_0\mc{P}$ via planar tangles.

\begin{lem}\label{composing_cyclic_derivative}
For $g\in Gr_0\mc{P}$, $x\in \mc{P}_{n,\pm}$, and $1\leq i \leq 2n$, consider the tangle
\[
\begin{tikzpicture}[thick, scale=.5]
\draw[Box] (0, 3) rectangle (2,4);
\node at (1,3.5) {$g$};
\node[marked, scale=.8, above left] at (0,4) {};
\draw[thickline] (0.5, 4) --++(0,1.5);
\draw[thick] (1, 4) arc(180:0:1) --+(0,-1) arc (0:-90: 1) arc(90:180:0.5);
\draw[thickline] (1.5,4) arc(180:0:.5) --++(0,-1) arc(0:-90:.5cm) --++(-2,0) arc(-90:-180:.5cm) --++(0,2.5);

\draw[Box] (-2,0.5) rectangle (4,1.5);
\node at (1,1) {$x$};
\node[marked, scale=.8, above left] at (-2,1.5) {};
\draw[thickline] (-1.4,1.5) --++ (0,4);
\draw[thickline] (3.5,1.5) --++ (0,4);
\node[right] at (4,0.8) {,};
\end{tikzpicture}
\]
where the $i$th boundary point of $x$ is connected with $g$ and we sum over all choices of boundary points of $g$. Then the image of the output of this tangle under $\hat{c}$ is the same as $\hat{c}(x)$ except with each monomial $C_{e_1}\cdots C_{e_{2n}}$ changed to $C_{e_1}\cdots (\D_{e_i} \hat{c}(g))\cdots C_{e_{2n}}$.
\end{lem}
\begin{proof}
We prove this result for the corresponding tangle on $Gr_0\mc{P}^\Gamma$, so that it then holds via our embedding $Gr_0\mc{P}\hookrightarrow Gr_0\mc{P}^\Gamma$. Suppose $w=e_1\cdots e_{2n}$ and $u=f_1\cdots f_{2m}$, are loops. Then
\[
\begin{tikzpicture}[thick, scale=.5]
\draw[Box] (0, 3) rectangle (2,4);
\node at (1,3.5) {$u$};
\node[marked, scale=.8, above left] at (0,4) {};
\draw[thickline] (0.5, 4) --++(0,1.5);
\draw[thick] (1, 4) arc(180:0:1) --+(0,-1) arc (0:-90: 1) arc(90:180:0.5);
\draw[thickline] (1.5,4) arc(180:0:.5) --++(0,-1) arc(0:-90:.5cm) --++(-2,0) arc(-90:-180:.5cm) --++(0,2.5);

\draw[Box] (-2,0.5) rectangle (4,1.5);
\node at (1,1) {$w$};
\node[marked, scale=.8, above left] at (-2,1.5) {};
\draw[thickline] (-1.4,1.5) --++ (0,4);
\draw[thickline] (3.5,1.5) --++ (0,4);
\node[right] at (4,2) {$\displaystyle =\sum_{j=1}^{2m} \delta_{f_j=e_i^\circ} \sigma(e_i^\circ) e_1\cdots e_{i-1}\left[\sigma(f_{j+1})^2 f_{j+1}\cdots  \sigma(f_{2m})^2 f_{2m}f_1\cdots f_{j-1}\right] e_{i+1}\cdots e_{2n}.$};
\end{tikzpicture}
\]
The image of this under $\hat{c}$ is precisely $C_{e_1}\cdots C_{e_{i-1}} \left[\D_{e_i} \hat{c}(u)\right] C_{e_{i+1}}\cdots C_{e_{2n}}$. Using the multilinearity of this and the tangle with respect to $u$ and $w$, we obtain the result for general $g$ and $x$.
\end{proof}

This lemma tells us that
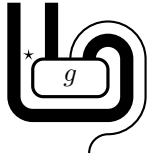
\[
\begin{tikzpicture}[thick, scale=.5]
\draw[Box] (0, 3) rectangle (2,4);
\node at (1,3.5) {$g$};
\node[marked, scale=.8, above left] at (0,4) {};
\draw[thickline] (0.5, 4) --++(0,1.5);
\draw[thick] (1, 4) arc(180:0:1) --+(0,-1) arc (0:-90: 1) arc(90:180:0.5);
\draw[thickline] (1.5,4) arc(180:0:.5) --++(0,-1) arc(0:-90:.5cm) --++(-2,0) arc(-90:-180:.5cm) --++(0,2.5);

\end{tikzpicture}
\]
can be thought of an $|E|$-tuple whose components are indexed by how we label the bottom string, and whose image under $\hat{c}$ is the $c$-cyclic gradient of $\hat{c}(g)$, $\D_c \hat{c}(g)$. That is, the $|E|$-tuple is $\D_c g$.

Identify $Gr_0\mc{P}^\Gamma$ with a subspace of its dual via the pairings
	\begin{equation}\label{sum_pairing}
		\sum_{v\in V} \left[\<f^*,\ \cdot\ \>_{\mc{P}^\Gamma}\right](v)\colon Gr_0\mc{P}^\Gamma\to\C,\qquad f\in Gr_0[[\mc{P}^\Gamma]].
	\end{equation}
Given a linear functional $\psi$ on $M$, $\psi\circ\hat{c}$ is a linear functional on $Gr_0\mc{P}^\Gamma$ and so by duality there is an element $f\in Gr_0[[\mc{P}^\Gamma]]$ so that
\[
\psi\circ\hat{c}(x)= \sum_{v\in V} \left[\<f^*,x\>_{\mc{P}^\Gamma}\right](v).
\]

\begin{lem}\label{S-D_planar_tangle_RHS}
Given a linear functional $\psi\colon M\to\C$, suppose the element $f\in Gr_0[[\mc{P}^\Gamma]]$ associated to $\psi$ as above belongs to the subspace $Gr_0[[\mc{P}]]$. Then for $x\in Gr_0\mc{P}$ embedding as $\sum_{u\in L} \beta_x(u) u\in Gr_0\mc{P}^\Gamma$ we have
\begin{equation}\label{tensor}
\begin{tikzpicture}[thick, scale=.5, baseline]
\draw[Box] (7.2,0) rectangle (11.8,1);
\node at (9.5,0.5) {$x$};
\node[marked, scale=.8, above left] at (7.2,1) {};
\draw[thick] (7.5,1) --++ (0, 1.5) arc(180:90:0.5) --++ (1.35,0) arc(90:0:0.5) --++ (0,-1.5);
\draw[thickline] (8.7,1) --++ (0,0.5);
\draw[thickline] (11,1) --++ (0,0.5);
\draw[Box] (7.9,1.5) rectangle (9.5, 2.5);
\node at (8.7,2) {$f$};
\node[marked, scale=.8, above left] at (7.9,2.5) {};
\draw[Box] (10.2,1.5) rectangle (11.8,2.5);
\node at (11,2) {$f$};
\node[marked, scale=.8, above left] at (10.2,2.5) {};
\node[right] at (12,0.5) {$\displaystyle = \psi\otimes\psi^{op}\left(\sum_{ eu\in L} \frac{1}{V(e)}\beta_x(eu) \partial_e \hat{c}(u) \right)$, };
\end{tikzpicture}
\end{equation}
where on the left we sum over the choices of the right-most endpoint of the string connecting $x$ to itself, and $V(e)\in \N$ is $|V_+|$ if $e\in E_+$ and $|V_-|$ otherwise.
\end{lem}
\begin{proof}
We first claim that
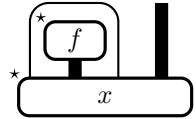
\[
\begin{tikzpicture}[thick, scale=.5]
\draw[Box] (7.2,0) rectangle (11.8,1);
\node at (9.5,0.5) {$x$};
\node[marked, scale=.8, above left] at (7.2,1) {};
\draw[thick] (7.5,1) --++ (0, 1.5) arc(180:90:0.5) --++ (1.35,0) arc(90:0:0.5) --++ (0,-1.5);
\draw[thickline] (8.7,1) --++ (0,0.5);
\draw[thickline] (11,1) --++ (0,2);
\draw[Box] (7.9,1.5) rectangle (9.5, 2.5);
\node at (8.7,2) {$f$};
\node[marked, scale=.8, above left] at (7.9,2.5) {};
\end{tikzpicture}
\]
embeds as $(\psi\otimes 1)(\sum_{eu\in L} \beta_x(eu) \partial_e \hat{c}(u))$ under $\hat{c}$. Indeed, let $e_1 u=e_1 e_2\cdots e_{2n}\in L$. Then this tangle evaluated at $e_1u$ instead of $x$ yields
	\begin{align*}
		\sum_{j=2}^{2n} \delta_{e_j=e_1^\circ} \sigma(e_1) [\<f^*, e_2\cdots e_{j-1}\>_{\mc{P}^\Gamma}](t(e_1)) e_{j+1}\cdots e_{2n}.
	\end{align*}
(We note that if $e_j=e_1^\circ$ then $e_2\cdots e_{j_1}$ and $e_{j+1}\cdots e_{2n}$ are indeed loops).

Now, since $\< f^*, e_2\cdots e_{j-1}\>_{\mc{P}^\Gamma}$ is supported only on $t(e_1)=s(e_2)$, we have $[\<f^*,e_2\cdots e_{j-1}\>_{\mc{P}^\Gamma}](t(e_1)) = \psi(\hat{c}(e_2\cdots e_{j_1}))$. Consequently the image of the above expression under $\hat{c}$ is
\begin{align*}
\sum_{j=2}^{2n} \delta_{e_j=e_1^\circ} \sigma(e_1) \psi(\hat{c}(e_2\cdots e_{j-1})) \hat{c}(e_{j+1}\cdots e_{2n}) = (\psi\otimes 1)(\partial_{e_1}\hat{c}(e_2\cdots e_{2n})) = (\psi\otimes 1)(\partial_{e_1} u).
\end{align*}
Summing over general $eu\in L$ yields the claim for $x$.

Now, for $a\in Gr_0^+\mc{P}$ we have that $\<f^*,a\>_{\mc{P}^\Gamma}$ is the function supported on $V_+$ with constant value of $\<f^*,a\>_{\mc{P}}$. Hence $\psi(\hat{c}(a))=|V_+| \<f^*,a\>_{\mc{P}}$, or
\[
\begin{tikzpicture}[thick, scale=.5]
\draw[Box] (0,0) rectangle (2,1);
\node at (1,0.5) {$a$};
\node[marked, scale=.8, above left] at (0,1) {};
\draw[thickline] (1,1) --++ (0,0.5);
\draw[Box] (0,1.5) rectangle (2, 2.5);
\node at (1,2) {$f$};
\node[marked, scale=.8, above left] at (0,2.5) {};
\node[right] at (2,1.25) {$\displaystyle =\frac{1}{|V_+|} \psi(\hat{c}(a)),$};
\end{tikzpicture}
\]
where the planar tangle is occurring in $\mc{P}$. Similarly for $a\in Gr_0^-\mc{P}$. Applying this to the output of the tangle in the first claim yields (\ref{tensor}) once we note that the components of $x$ in $Gr_0^{\pm}\mc{P}$ embed as $\sum_{eu\in L_{\pm}}\beta_x(eu) eu \in Gr_0^{\pm} \mc{P}^\Gamma$, respectively.
\end{proof}

\begin{rem}\label{inner_product_comparison}
The element associated to the free quasi-free state $\varphi$ by (\ref{sum_pairing}) is $TL_\infty$, which we note is distinct from $\overline{TL}_\infty$, the element associated to it via the pairing $\<f^*,\ \cdot\ \>_{\mc{P}}$ on $Gr_0\mc{P}$. This difference is simply a consequence of the relationship between these two pairings for elements of $Gr_0\mc{P}$:
\[
\sum_{v\in V}[\<f^*, x\>_{\mc{P}^\Gamma}](v) = |V_+|\<f^*_+, x_+\>_\mc{P} + |V_-| \<f^*_-, x_-\>_{\mc{P}} \qquad \text{for }f \in Gr_0[[\mc{P}]],\ x\in Gr_0\mc{P}.
\]
\end{rem}

%%%%%%%%%%%%%%%%%%%%%%%%
%	Formal power series and Banach algebras

\subsection{Formal power series and Banach algebras}

Recall from \cite{N13} that on $\P$ we considered two Banach norms for an $R>0$. For
\begin{align*}
Q=\sum_{n\geq 0} \sum_{e_1,\ldots, e_n\in E} \beta_Q(e_1,\ldots, e_n) X_{e_1}\cdots X_{e_n},\qquad \beta_Q(e_1,\ldots, e_n)\in \C
\end{align*}
we defined the Banach norm
\begin{align*}
\|Q\|_R := \sum_{n\geq 0} \sum_{e_1,\ldots, e_n\in E} |\beta_Q(e_1,\ldots, e_n)| R^n.
\end{align*}
We denote the closure $\overline{\P}^{\|\cdot\|_R}$ by $\P^{(R)}$. 

\begin{lem}\label{analytically_free}
Let $R>\max_{j} \|X_j\|$ and suppose the coefficients $\beta_Q(e_1,\ldots, e_n)\in \C$, for $n\geq 0$ and $e_1,\ldots, e_n\in E$, satisfy
	\begin{align*}
		\sum_{n\geq 0} \sum_{e_1,\ldots, e_n\in E} |\beta_Q(e_1,\ldots, e_n)| R^n<\infty.
	\end{align*}
Then
	\begin{align*}
		Q:=\sum_{n\geq 0} \sum_{e_1,\ldots, e_n\in E} \beta_Q(e_1,\ldots, e_n) X_{e_1}\cdots X_{e_n}
	\end{align*}
is an element of $M$ and $Q=0$ if and only if every coefficient $\beta_Q(e_1,\ldots, e_n)=0$.
\end{lem}
\begin{proof}
The hypothesis on the coefficients implies $\|Q\|_R<\infty$ and that
	\begin{align*}
		Q_k:=\sum_{0\leq n \leq k} \sum_{e_1,\ldots, e_n\in E} \beta_Q(e_1,\ldots, e_n) X_{e_1}\cdots X_{e_n}
	\end{align*}
converge to $Q$ in the $\|\cdot\|_R$-norm. As the $Q_k$ are polynomials in the $X_e$, they lie in $M$. Since the $\|\cdot\|_R$-norm dominates the operator norm by our hypothesis on $R$, we then see that $Q\in \overline{\P}^{\|\cdot\|}\subset M$.

Now, suppose $Q=0$. The rest of the proof follows \textit{mutatis mutandis} from Lemma 37 of \cite{D10} once we note that the free difference quotients $\{\delta_e\}_{e\in E}$ are closable. As each $\delta_e$ is linear combinations of the $\partial_{u(e)}$ (using the fact that $\frac{2}{1+A}$ is invertible), it suffices to show that each $\partial_{u(e)}$ is closable. This can be easily checked using Equation (10) in Corollary 2.4 from \cite{N13}.
\end{proof}

\begin{rem}\label{conjugate_variables_suffice}
The second part of the previous lemma is really asserting that the generators are analytically free. 

We also note that the closability of the $\{\delta_e\}_{e\in E}$ relies only on the existence of \emph{conjugate variables} to the $\{\partial_{u(e)}\}_{e\in E}$; that is, elements $\{\xi_{u(e)}\}_{e\in E}\subset L^2(\P, \varphi)$ satisfying
	\begin{align*}
		\varphi(\xi_{u(e)} Q) = \varphi\otimes\varphi^{op}(\partial_{u(e)} Q).
	\end{align*}
Indeed, viewing $\partial_{u(e)}\colon L^2(\P,\varphi)\to L^2(\P\otimes\P^{op}, \varphi\otimes\varphi^{op})$ as a densely defined map, $\xi_{u(e)} = \partial_{u(e)}^*(1\otimes 1)$. Then Equation (10) in Corollary 2.4 of \cite{N13} holds when $X_{u(e)}$ is replaced with $\xi_{u(e)}$, and hence $\partial_{u(e)}$ is closable.
\end{rem}

We note
\begin{align*}
\|X_e\| = \|\hat{\ell}(u(e))+\hat{\ell}(u(e))^*\| \leq \frac{2}{\sqrt{\sigma(e)+\sigma(e^\circ)}}(\|e\|+\|e^\circ\|) <2(1+\delta^\frac{1}{4}).
\end{align*}
Thus, in light of Lemma \ref{analytically_free}, we will usually consider $R\geq 2(1+\delta^\frac{1}{4})$ so that $\P^{(R)}\subset M$. In fact, due to hypotheses of the free transport theorems (\emph{cf.} Theorem 3.17 in \cite{N13} for example) we will usually restrict ourselves to
\begin{align*}
R\geq 4\delta^\frac{1}{2} >4\sqrt{\|A\|},
\end{align*}
where we have used (\ref{norm_of_A}).

We let $\P^{(R)}_\varphi$ denote the intersection of $\P^{(R)}$ with $M_\varphi$, the centralizer of $M$ with respect to $\varphi$ (i.e. the elements fixed under the modular automorphism group $\{\sigma_t^\varphi\}_{t\in\R}$).

Writing $Q=\sum_{n\geq 0} \pi_n(Q)$ where $\pi_n$ is the projection onto monomials of degree $n$, we also defined in \cite{N13} the Banach norm
\begin{align*}
\|Q\|_{R,\sigma}:= \sum_{n\geq 0} \sup_{k_n\in \Z} \| \rho^{k_n}(\pi_n(Q))\|_R,
\end{align*}
where $\rho\colon \P\to\P$ is defined by
	\begin{align*}
		\rho(X_{e_1}\cdots X_{e_n}):=\sigma_{-i}^\varphi(X_{e_n}) X_{e_1}\cdots X_{e_{n-1}},
	\end{align*}
and $\rho(a)=a$ for $a\in \C$. For a polynomial $Q\in \P$, $\rho(Q)$ is a called a \emph{$\sigma$-cyclic rearrangement} of $Q$. The tangle induced by $\rho$ on $Gr_0\mc{P}^\Gamma$ is the identity tangle but with the last string rotated clockwise around to the leftmost boundary point of the output disc. Equivalently, the tangle shifts the distinguished interval to the adjacent interval in the counter-clockwise direction.

Let $\P^{finite}=\{Q\in \P\colon \|Q\|_{R,\sigma}<\infty\}$, then it is easy to see that $\P\cap M_\varphi\subset \P^{finite}$ and we let $\P^{(R,\sigma)}=\overline{\P^{finite}}^{\|\cdot\|_{R,\sigma}}$. Observe that $\P^{(R,\sigma)}\subset \P^{(R)}\subset M$ since the $\|\cdot\|_R$-norm is dominated by the $\|\cdot\|_{R,\sigma}$-nrom. We also denote $\P^{(R,\sigma)}_\varphi=\P^{(R,\sigma)}\cap M_\varphi$ and further denote by $\P^{(R,\sigma)}_{c.s.}$ the elements in $\P^{(R,\sigma)}$ which are fixed under $\rho$. Such elements are called \emph{$\sigma$-cyclically symmetric} and have the same norm with respect to $\|\cdot\|_R$ and $\|\cdot\|_{R,\sigma}$.

Via the embedding $\hat{c}$, the norms $\|\cdot\|$, $\|\cdot\|_R$, and $\|\cdot\|_{R,\sigma}$ induce norms on $Gr_0 \mc{P}^\Gamma$, which we denote in the same way, and maps $\sigma_z^\varphi$, $z\in\C$, and $\rho$ induce a maps on $Gr_0 \mc{P}^\Gamma$, again still denoted in the same way. Let
\begin{align*}
(Gr_0 \mc{P}^\Gamma)^{(R)}&:=\overline{Gr_0 \mc{P}^\Gamma}^{\|\cdot\|_R},\qquad \text{ and}\\
(Gr_0 \mc{P}^\Gamma)^{(R,\sigma)} &:= \overline{Gr_0 \mc{P}^\Gamma}^{\|\cdot\|_{R,\sigma}}
\end{align*}
(we will see below that $\|w\|_{R,\sigma}<\infty$ for all $w\in Gr_0 \mc{P}^\Gamma$). We similarly define $(Gr_0\mc{P})^{(R)}$ and $(Gr_0\mc{P})^{(R,\sigma)}$.

$(Gr_0 \mc{P}^\Gamma)^{(R)}$ may be thought of the subalgebra of $Gr_0[[\mc{P}^\Gamma]]$ of absolutely convergent power series on loops with radii of convergence at least $R$, where a loop of length $2n$ is given degree $2n$ (modulo the constants involved in translating from $X$ to $C$). Similarly, $(Gr_0 \mc{P}^\Gamma)^{(R,\sigma)}$ may be thought of as the subalgebra of $Gr_0[[\mc{P}^\Gamma]]$ of absolutely convergent power series on the loops so that every rotation of their support loops has a radius of convergence of at least $R$. We also use the subscripts $\varphi$ and $c.s.$ to denote the corresponding subspaces.

We make the following observations for a loop $e_1\cdots e_{2n}\in L_{n,\pm}$:
\begin{align*}
\sigma_{-i}^\varphi(e_1\cdots e_{2n}) = \left(\prod_{l=1}^{2n} \frac{\mu(t(e_l))}{\mu(s(e_l))}\right) e_1\cdots e_{2n} = e_1\cdots e_{2n},
\end{align*}
and for $1\leq k< 2n$
\begin{align}\label{rho_on_loops}
\rho^k(e_1\cdots e_{2n}) &= \left(\prod_{l=2n-k+1}^{2n} \frac{\mu(t(e_l))}{\mu(s(e_l))}\right) e_{2n-k+1}\cdots e_{2n}e_1\cdots e_{2n-k}\nonumber\\
&= \frac{\mu(t(e_{2n}))}{\mu(s(e_{2n-k+1}))} e_{2n-k+1}\cdots e_{2n} e_1\cdots e_{2n-k}.
\end{align}
Note that for $e\in E$
\begin{align*}
\| C_e\|_R =\left\|\frac{\sqrt{\sigma(e)+\sigma(e^\circ)}}{2}(X_e\pm i X_{e^\circ})\right\|_R \leq \sqrt{1+\delta^{1/2}} R,
\end{align*}
where we used the bound $\frac{\mu(v)}{\mu(v')}<\delta$ for adjacent vertices $v,v'\in V$. Thus for $w= \sum_{u\in L_{n,\pm}} \beta_w(u) u \in \mc{P}^\Gamma_{n,\pm}$ we have the bound
\begin{align*}
\|w\|_R \leq \sum_{u\in L_{n,\pm}} |\beta_w(u)|(1+\delta^{1/2})^n R^{2n},
\end{align*}
and using (\ref{rho_on_loops}) we obtain
\begin{align*}
\|w\|_{R,\sigma} \leq \Delta \sum_{u\in L_{n,\pm}} |\beta_w(u)|(1+\delta^{1/2})^n R^{2n},
\end{align*}
where $\Delta= \max_{v,v'\in V} \frac{\mu(v)}{\mu(v')}<\infty$. In particular, for any $w\in Gr_0 \mc{P}^\Gamma$, $\|w\|_{R,\sigma}<\infty$.

%%%%%%%%%%%%%%%%%%%%%%%%
% The Schwinger-Dyson equation

\subsection{The Schwinger-Dyson planar tangle}

Let $\psi\colon M\to \C$ be a state on the free Araki-Woods factor $M$ and let $V\in \P^{(R,\sigma)}_{c.s.}$, with $R\geq 4\delta^\frac{1}{2}$. Then $\psi$ is said to satisfy \emph{the Schwinger-Dyson equation with potential $V$} if
\begin{align*}
\psi( \D V \# Q)=\psi\otimes\psi^{op}\otimes\text{Tr}(\J_\sigma Q)\qquad Q\in \P^N.
\end{align*}

\begin{rem}
This equation implies that the conjugate variables to the $\{\partial_{u(e)}\}_{e\in E}$ are $\{\D_{u(e)}V\}_{e\in E}$. Hence Remark \ref{conjugate_variables_suffice} implies that operators whose joint law satisfies the Schwinger-Dyson equation with potential $V$ are analytically free.
\end{rem}

Using (\ref{linear_relation_derivative1}) and (\ref{linear_relation_derivative2}) the Schwinger-Dyson equation is equivalent to
\begin{align}\label{circular_schwinger-dyson}
\psi(\D_c V \# Q) = \psi\otimes\psi^{op}\otimes\text{Tr}(\J_c Q)\qquad Q\in \P^N.
\end{align}
The solution $\psi$ is a \emph{free Gibbs state} with potential $V$ and is often denoted $\varphi_V$.

\begin{lem}\label{S-D_planar_tangle_lemma}
Let $\psi$ be a free Gibbs state with potential $V$. Assume that $V=\hat{c}(v)$ for some $v\in (Gr_0\mc{P})^{(R,\sigma)}_{c.s.}$, and that the element $f\in Gr_0[[\mc{P}^\Gamma]]$ associated to $\psi$ by the duality in (\ref{sum_pairing}) satisfies $f\in Gr_0[[\mc{P}]]$. Then the following equivalence of planar tangles holds:
\begin{equation}\label{S-D_planar_tangle}
\begin{tikzpicture}[thick, scale=.5, baseline]
\draw[Box](1.5,2.15) rectangle (3.5,3.15);
\node at (2.5,2.65) {$v$};
\node[marked, scale=.8, above left] at (1.5,3.15) {};
\draw[thickline] (2,3.15) --++ (0,1);
\draw[thick] (2.5,3.15) arc(180:0:0.9) --++ (0,-1) arc(0:-90:0.9) --++ (-0.65,0) arc(90:180:0.25);
\draw[thickline] (3,3.15) arc(180:0:0.4) --++ (0,-1) arc(0:-90:0.4) --++ (-2,0) arc(270:180:0.4) --++ (0,2);
\draw[Box] (2.2,0) rectangle (5.2,1);
\node at (3.7, 0.5) {$x$};
\node[marked, scale=.8, above left] at (2.2,1) {};
\draw[thickline] (4.8,1) --++(0,3.15);

\draw[Box] (0.6,4.15) rectangle (5.2,5.15);
\node at (2.9,4.65) {$f$};
\node[marked, scale=.8, above left] at (0.6,5.15) {};
\node at (6.2,0.5) {$=$};
\draw[Box] (7.2,0) rectangle (11.8,1);
\node at (9.5,0.5) {$x$};
\node[marked, scale=.8, above left] at (7.2,1) {};
\draw[thick] (7.5,1) --++ (0, 1.5) arc(180:90:0.5) --++ (1.35,0) arc(90:0:0.5) --++ (0,-1.5);
\draw[thickline] (8.7,1) --++ (0,0.5);
\draw[thickline] (11,1) --++ (0,0.5);
\draw[Box] (7.9,1.5) rectangle (9.5, 2.5);
\node at (8.7,2) {$f$};
\node[marked, scale=.8, above left] at (7.9,2.5) {};
\draw[Box] (10.2,1.5) rectangle (11.8,2.5);
\node at (11,2) {$f$};
\node[marked, scale=.8, above left] at (10.2,2.5) {};
\node[below right] at (12,0.5) {,};
\end{tikzpicture}
\end{equation}
where on the left we sum over the boundary points of $v$ which are connected to $x$, and on the right we sum over the positions of the right endpoint of the string.
\end{lem}
\begin{proof}
Let
\[
\begin{tikzpicture}[thick, scale=.5]
\node[left] at (1,0.5) {$y=$};
\draw[Box](1.5,2.15) rectangle (3.5,3.15);
\node at (2.5,2.65) {$v$};
\node[marked, scale=.8, above left] at (1.5,3.15) {};
\draw[thickline] (2,3.15) --++ (0,1);
\draw[thick] (2.5,3.15) arc(180:0:0.9) --++ (0,-1) arc(0:-90:0.9) --++ (-0.65,0) arc(90:180:0.25);
\draw[thickline] (3,3.15) arc(180:0:0.4) --++ (0,-1) arc(0:-90:0.4) --++ (-2,0) arc(270:180:0.4) --++ (0,2);
\draw[Box] (2.2,0) rectangle (5.2,1);
\node at (3.7, 0.5) {$x$};
\node[marked, scale=.8, above left] at (2.2,1) {};
\draw[thickline] (4.8,1) --++(0,3.15);

\node[below right] at (5.2,0.5) {,};
\end{tikzpicture}
\]
and suppose $x$ embeds as $\sum_{u\in L}\beta_x(u) u\in Gr_0\mc{P}^\Gamma$.
Then by Remark \ref{inner_product_comparison} and Lemma \ref{composing_cyclic_derivative}
\begin{align*}
\<f^*,y\>_\mc{P}&= \sum_{v\in V_+} \frac{1}{|V_+|}[\<f^*,y_+\>_{\mc{P}^\Gamma}](v) + \sum_{v\in V_-}\frac{1}{|V_-|}[\<f^*, y_-\>_{\mc{P}^\Gamma}](v)\\
&= \psi\circ\hat{c}\left( \frac{1}{|V_+|} y_+ + \frac{1}{|V_-|} y_-\right)= \sum_{eu\in L} \frac{\beta_x(eu)}{V(e)}\psi(\D_e \hat{c}(v) \cdot \hat{c}(u)),
\end{align*}
where $V(e)=|V_+|$ if $e\in E_+$ and $V(e)=|V_-|$ otherwise. Next applying (\ref{circular_schwinger-dyson}) yields
\begin{align*}
\<f^*,y\>_\mc{P} = \sum_{eu\in L} \frac{\beta_x(eu)}{V(e)} \psi\otimes\psi^{op}(\partial_e \hat{c}(u)),
\end{align*}
which is equivalent to the right-hand side of (\ref{S-D_planar_tangle}) by Lemma \ref{S-D_planar_tangle_RHS}.
\end{proof}

\begin{defi}
For $v\in (Gr_0\mc{P})^{(R,\sigma)}_{c.s.}$, we say $f\in Gr_0[[\mc{P}]]$ \emph{satisfies the Schwinger-Dyson planar tangle with potential $v$} if (\ref{S-D_planar_tangle}) holds for all $x\in Gr_0\mc{P}$.
\end{defi}

Recall that in \cite{N13} we considered the following potential
\begin{align*}
V_0=\frac{1}{2}\sum_{e,f\in E} \left[\frac{1+A}{2}\right]_{ef} X_f X_e,
\end{align*}
which satisfied $\D V_0=X$. The (unique) free Gibbs state with potential $V_0$ is the vacuum state $\varphi$. Furthermore, by Theorem 2.12 in \cite{N13} there is a unique free Gibbs state with potential $V$ when $\|V - V_0\|_{R,\sigma}$ is sufficiently small.

Rewriting $V_0$ in terms of the $C_e$ via (\ref{linear_relation}) and using (\ref{linear_relation2}) yields
\begin{align*}
V_0=\frac{1}{2} \sum_{e\in E} \sigma(e) C_e C_{e^\circ},
\end{align*}
and $\D_c V_0=U\# \D V_0=C$. Observe that $V_0=\hat{c}(v_0)$ where $v_0\in Gr_0\mc{P}$ is the sum of the $1$-box Temperley-Lieb diagrams
\[
\begin{tikzpicture}[thick, scale=.5]
\node[left] at (-0.6,0.5) {$v_0=$};
\node[left] at (0,0.5) {$\frac{1}{2}$};
\draw[Box](0,0) rectangle (2,1);
\node[marked, scale=.8, above left] at (0,1) {};
\draw[thick, fill=gray] (0.5,0.96) arc(180:360:0.5);
\node at (2.5,0.5) {$+$};
\node[left] at (3.5,0.5) {$\frac{1}{2}$};
\draw[Box, fill=gray] (3.5,0) rectangle (5.5,1);
\node[marked, scale=.8, above left] at (3.5,1) {};
\draw[thick, fill=white] (4,0.96) arc(180:360:0.5);
\node[below right] at (5.5,0.5) {,};
\end{tikzpicture}
\]
which embeds as $\frac{1}{2} \sum_{e\in E} \sigma(e^\circ) e e^\circ \in Gr_0 \mc{P}^\Gamma$. Since $\varphi$ satisfies with Schwinger-Dyson equation with potential $V_0$, and $TL_\infty$ is the element associated to it by the duality in (\ref{sum_pairing}), we know $TL_\infty$ satisfies the Schwinger-Dyson planar tangle with potential $v_0$ by the previous lemma.

However, this is true by visual inspection within the context of the planar algebra: note that
\[
	\begin{tikzpicture}[thick, scale=.5]
	\draw[Box] (0, 0) rectangle (2,1);
	\node at (1,0.5) {$v_0$};
	\node[marked, scale=.8, above left] at (0,1) {};
	\draw[thickline] (0.5, 1) --++(0,1.5);
	\draw[thick] (1, 1) arc(180:0:1) --+(0,-1) arc (0:-90: 1) arc(90:180:0.5);
	\draw[thickline] (1.5,1) arc(180:0:.5) --++(0,-1) arc(0:-90:.5cm) --++(-2,0) arc(-90:-180:.5cm) --++(0,2.5);
	\node at (4,0.5) {$=$};
	\draw[Box] (5,0) rectangle (7,1);
	\node[marked, scale=.8, above left] at (5,1) {};
	\draw[very thick, fill=gray] (6,1) --++ (.75,0) arc(90:0:.25) --++ (0,-0.5) arc(0:-90:0.25) --++ (-0.75,0);
	\draw[thick] (6,0) --++ (0,1);
	\node at (7.5,0.5) {$+$};
	\draw[Box, fill=gray] (8,0) rectangle (10,1);
	\node[marked, scale=.8, above left] at (8,1) {};
	\draw[very thick, fill=white] (9,1) --++ (.75,0) arc(90:0:.25) --++ (0,-0.5) arc(0:-90:0.25) --++ (-0.75,0);
	\draw[thick] (9,0) --++ (0,1);
	\node[right] at (10,0.3) {.};
	\end{tikzpicture}
\]
Hence the Schwinger-Dyson planar tangle holds simply by following the leftmost string attached to $x$ through the diagrams in $TL_\infty$.

In  Section \ref{constructing_transport_element_section}, we construct elements $TL_\infty^{(v)}\in Gr_0[[\mc{P}]]$ which satisfy the Schwinger-Dyson planar tangle for potentials $v$ close to $v_0$ with respect to the $\|\cdot\|_{R,\sigma}$-norm. Our convention will be to denote the difference by $w=v-v_0$. We will also construct an embedding of $Gr_0\mc{P}^\Gamma$ into $M$ taking the edges $e\in E$ to non-commutative random variables whose joint law with respect to $\varphi$ is the free Gibbs state with potential $V=\hat{c}(v)$.

%%%%%%%%%%%%%%%%%%%%%%%%%%%%%%%%%%%%%%%%%%%%%%%%%%%%%
% Transport %
%%%%%%%%%%%%%%%%%%%%%%%%%%%%%%%%%%%%%%%%%%%%%%%%%%%%%

\section{Free Transport}

For the remainder of the paper we fix $R'>R\geq 4\delta^\frac{1}{2}$. The constants obtained in the following will depend only on $R$, $R'$, $|E|$, and $\|A\|$.

%%%%%%%%%%%%%%%%%%%%%%%%%%%%%%%%%%%%%%%%%%%%%%%%%%%%%%%%
%Constructing the transport element

\subsection{Constructing the transport element}\label{constructing_transport_element_section}

The main theorem of \cite{N13} showed that if $Z$ is an $N$-tuple of random variables in some non-commutative probability $(L,\psi)$ whose joint law $\psi_Z$ is the free Gibbs state with potential $V$, and $\|V-V_0\|_{R,\sigma}$ is sufficiently small, then $(W^*(Z),\psi)\cong (W^*(X),\varphi)$ and the isomorphism is state-preserving. Stated more succinctly, the theorem gives $W^*(\varphi_V)\cong W^*(\varphi_{V_0})$ for $\|V-V_0\|_{R,\sigma}$ sufficiently small. In this section we will show that if $v\in (Gr_0 \mc{P})^{(R,\sigma)}_{c.s.}$ with $\|v-v_0\|_{R,\sigma}$ is sufficiently small, then there is an element satisfying the Schwinger-Dyson planar tangle with potential $v$.

Recall that the map $\mathscr{N}\colon\mathscr{P}\rightarrow\mathscr{P}$ is defined by multiplying a monomial of degree $n$ by $n$, and $\Sigma$ is its inverse on monomials of degree one or higher. Also, $\mathscr{S}\colon \mathscr{P}\rightarrow\mathscr{P}$ averages a monomial over its $\sigma$-cyclic rearrangements. These induce maps on $Gr_0 \mc{P}^\Gamma$, which we also denote $\mathscr{N}$, $\Sigma$, and $\mathscr{S}$:
	\begin{align*}
		\mathscr{N}(e_1\cdots e_{2n}) &= 2n e_1\cdots e_{2n},\\
		\Sigma(e_1\cdots e_{2n}) &= \frac{1}{2n} e_1\cdots e_{2n},\ \text{ and}\\
		\mathscr{S}(e_1\cdots e_{2n}) & = \frac{1}{2n}\sum_{k=1}^{2n} \rho^k(e_1\cdots e_{2n}),
	\end{align*}
or for $x\in \mc{P}_n\subset\mc{P}_n^\Gamma$
	\begin{align*}
		\mathscr{N}(x)&=2n x,\\
		\Sigma(x) &= \frac{1}{2n} x,\ \text{ and}\\
		\mathscr{S}(x) &= \frac{1}{2n}\sum_{k=1}^{2n} \rho^k(x).
	\end{align*}

\begin{lem}\label{F_is_a_tangle}
Let $w\in (Gr_0 \mc{P})^{(R'+1,\sigma)}_{c.s.}$ and denote $W:=\hat{c}(w)$. Consider the following map defined on $\{G\in\P^{(R',\sigma)}_{c.s.}\colon \|G\|_{R',\sigma}\leq 1\}$:
	\begin{align*}
		F(G)=&-W(C+\D_c\Sigma G) - \frac{1}{2} \sum_{e\in E} \sigma(e) \left(\D_e \Sigma G\right)\left(\D_{e^\circ}\Sigma G\right)\\
			&+ \sum_{m\geq 1} \frac{(-1)^{m+1}}{m} (1\otimes\varphi)\circ\text{Tr}\left( \left[U \frac{2A^{-1}}{1+A} U^T\right]^{-1} \J_c \D_c \Sigma G\# \left( \J_c C \# \J_c\D_c\Sigma G\right)^{m-1}\right)\\
			&+ \sum_{m\geq 1} \frac{(-1)^{m+1}}{m} (\varphi\otimes 1)\circ\text{Tr}\left( \left[U \frac{2A}{1+A} U^T\right]^{-1} \J_c \D_c \Sigma G \# \left( \J_c C \# \J_c\D_c\Sigma G\right)^{m-1}\right).
	\end{align*}
Consider the following planar tangles on $Gr_0 \mc{P}^\Gamma$:
	\begin{equation*}
	\begin{tikzpicture}[thick, scale=.5]
		\node[left] at (-1,2) {$T_1(g)=$};

		\draw[Box] (0, 3) rectangle (3,4); \node at (1.5,3.5) {$v_0+\Sigma g$}; \node[marked, scale=.8, above left] at (0,4) {};
		\draw[thickline] (0.75, 4) --++(0,1.5);
		\draw[thick] (1.5, 4) arc(180:0: 1.3cm and 1cm) --+(0,-1) arc (0:-90: 1.2 cm and 1cm) -- (2, 2) arc(90:180:.5cm);
		\draw[thickline] (2.25,4) arc(180:0: .6cm and .5cm) --++(0,-1) arc(0:-90:.5cm) --++(-3,0) arc(-90:-180:.5cm) --++(0,2.5);
		
		\node at (5.25,3.75) {$\cdots$};
			
		\draw[Box] (7, 3) rectangle (10,4); \node at (8.5,3.5) {$v_0+\Sigma g$}; \node[marked, scale=.8, above left] at (7,4) {};
		\draw[thickline] (7.75, 4) --++(0,1.5);
		\draw[thick] (8.5, 4) arc(180:0: 1.3cm and 1cm) --++(0,-1) arc (0:-90: 1.2 cm and 1cm) -- (8.75, 2) arc(90:180:.5cm);
		\draw[thickline] (9.25,4) arc(180:0: .6cm and .5cm) --++(0,-1) arc(0:-90:.5cm) --++(-3,0) arc(-90:-180:.5cm) --++(0,2.5);
		
		\draw[Box] (1,.5) rectangle (9,1.5); \node at (5,1) {$w$}; \node[marked,scale=.8, above left] at (1,1.5) {};
		
		\node[right] at (9,0.8) {,};
		
	\end{tikzpicture}
	\end{equation*}
where the number discs containing $v_0+\Sigma g$ varies according to the components of $w$ and for each such disc we sum over the boundary points connecting to $w$;
	\begin{equation*}
	\begin{tikzpicture}[thick, scale=.5]
		\node[left] at (-1,.5) {$T_2(g)=$};
		
		\draw[Box] (0,0) rectangle (2, 1); \node at (1,.5) {$\Sigma g$}; \node[marked, scale=.8, above left] at (0, 1) {};
		\draw[thickline] (0.5,1) --++(0,1.3);
		\draw[thick] (1,1) arc(180:0: 1.1cm and .9cm) --++(0,-1) arc(180:270: 1cm) --++(2.3,0) arc(270:360:1cm);
		\draw[thickline] (1.5,1) arc(180:0: .6 cm and .5cm) --++(0,-1) arc(0:-90:.5cm) --++(-2.2,0) arc(270:180:.5cm) --++(0,2.3);

		\draw[Box] (4.3,0) rectangle (6.3,1); \node at (5.3,.5) {$\Sigma g$}; \node[marked, scale=.8, above left] at (4.3,1) {};
		\draw[thickline] (4.8,1) --++(0,1.3);
		\draw[thick] (5.3,1) arc(180:0: 1.1cm and .9cm) --++(0,-1);
		\draw[thickline] (5.8,1) arc(180:0: .6 cm and .5cm) --++(0,-1) arc(0:-90:.5cm) --++(-2.2,0) arc(270:180:.5cm) --++(0,2.3);
		
		\node[right] at (7.5,0.3) {,};
	\end{tikzpicture}
	\end{equation*}
where in each disc we sum over the boundary point connecting to the other disc;
	\begin{equation*}
	\begin{tikzpicture}[thick, scale=.5]
		\node[left] at (1, 4.5) {$T_{3,m}(g)=$};
		
		\draw[Box] (0,0) rectangle (2,1); \node at (1,.5) {$\Sigma g$}; \node[marked, scale=.8, above left] at (0,1) {};
		\draw[thickline] (0.5,1) arc(0:60:.5cm) arc(240:90:.5cm) --++ (2.5,0) arc (90:0:.5cm) --++(0,-3.5);
		\draw[thick] (1,1) arc(180:0: 1.05cm and .9cm) --++(0,-1.3) arc(0:-90:.7cm) --++(-.05,0) arc(90:180:.7cm) --++(0,-1.05) arc(180:270:.5cm)--++(23.9,0) 		arc(-90:0:.5cm);
		\draw[thickline] (1.5,1) arc(180:0: .6 cm and .5cm) --++(0,-1) arc(0:-90:.5cm) --++(-2.4,0) arc(270:180:.5cm) --++(0,1) arc(180:120:.5cm) arc(-60:0:.5cm) --++ 	(0,1.5);
		\draw[thick] (-.05,2) arc(180:90: 1cm) --++(2.2,0) arc(90:0:1cm) --++(0,-2);
		
		\draw[Box] (5.5,0) rectangle (7.5,1); \node at (6.5,.5) {$\Sigma g$}; \node[marked, scale=.8, above left] at (5.5,1) {};
		\draw[thickline] (6,1) arc(0:60:.5cm) arc(240:90:.5cm) --++ (2.5,0) arc (90:0:.5cm) --++(0,-3.5);
		\draw[thick] (6.5,1) arc(180:0: 1.05cm and .9cm) --++(0,-1.3) arc(0:-90:.7cm) --++(-2.75,0) arc(270:180:1cm);
		\draw[thickline] (7,1) arc(180:0: .6 cm and .5cm) --++(0,-1) arc(0:-90:.5cm) --++(-2.4,0) arc(270:180:.5cm) --++(0,1) arc(180:120:.5cm) arc(-60:0:.5cm) --++ (0,1.5);
		\draw[thick] (5.45,2) arc(180:90: 1cm) --++(2.2,0) arc(90:0:1cm) --++(0,-2) arc(180:270:1cm);
		
		\draw[thick, dotted] (10.65,-1)--++(1,0);
		
		\node at (12.1,1) {$\cdots$};
		
		\draw[thick, dotted] (13.55,3) --++(-1,0);
		
		\draw[Box] (16,0) rectangle (18,1); \node at (17,.5) {$\Sigma g$}; \node[marked, scale=.8, above left] at (16,1) {};
		\draw[thickline] (16.5,1) arc(0:60:.5cm) arc(240:90:.5cm) --++ (2.5,0) arc (90:0:.5cm) --++(0,-3.5);
		\draw[thick] (17,1) arc(180:0: 1.05cm and .9cm) --++(0,-1.3) arc(0:-90:.7cm) --++(-2.75,0) arc(270:180:1cm)--++(0,2) arc(0:90:1cm);
		\draw[thickline] (17.5,1) arc(180:0: .6 cm and .5cm) --++(0,-1) arc(0:-90:.5cm) --++(-2.4,0) arc(270:180:.5cm) --++(0,1) arc(180:120:.5cm) arc(-60:0:.5cm) --++ 		(0,1.5);
		\draw[thick] (15.95,2) arc(180:90: 1cm) --++(2.2,0) arc(90:0:1cm) --++(0,-2);
		
		\draw[Box] (21.5,0) rectangle (23.5,1); \node at (22.5,.5) {$\Sigma g$}; \node[marked, scale=.8, above left] at (21.5,1) {};
		\draw[thickline] (22,1) arc(0:60:.5cm) arc(240:90:.5cm) --++ (2.5,0) arc (90:0:.5cm) --++(0,-3.5);
		\draw[thick] (22.5,1) arc(180:0: 1.05cm and .9cm) --++(0,-1.3) arc(0:-90:.7cm) --++(-2.75,0) arc(270:180:1cm);
		\draw[thickline] (23,1) arc(180:0: .6 cm and .5cm) --++(0,-1) arc(0:-90:.5cm) --++(-2.4,0) arc(270:180:.5cm) --++(0,1) arc(180:120:.5cm) arc(-60:0:.5cm) --++ (0,1.5);
		\draw[thick] (21.45,2) arc(180:90: 1cm) --++(3.1,0) arc(90:0:1cm) --++(0,-4.76);
		
		\draw[Box] (2.5,-2.6) rectangle (26,-1.6); \node at (14.25,-2.1) {$TL_\infty$}; \node[marked, scale=.8, above left] at (2.5,-1.6) {};
	
		\node[right] at (26.5,-2.3) {,};
	\end{tikzpicture}
	\end{equation*}
where there are exactly $m$ discs containing $\Sigma g$ and for each disc we sum over the two boundary points connecting to one of the other $m-1$ discs; and finally
	\begin{equation*}
	\begin{tikzpicture}[thick, scale=.5]
		\node[left] at (0, 6) {$T_{4,m}(g)=$};
		
		\draw[Box] (-1.1,3.4) rectangle (22.4,4.4); \node at (10.65,3.9) {$TL_\infty$}; \node[marked, scale=.8, above left] at (-1.1, 4.4) {};
		
		\draw[Box] (0,0) rectangle (2,1); \node at (1,.5) {$\Sigma g$}; \node[marked, scale=.8, above left] at (0,1) {};
		\draw[thickline] (0.5,1) arc(0:60:.5cm) arc(240:90:.5cm) --++ (2.5,0) arc (90:0:.5cm) --++(0,-3.5);
		\draw[thick] (1,1) arc(180:0: 1.05cm and .9cm) --++(0,-1.3) arc(0:-90:.7cm) --++(-3.5,0) arc(270:180:.7cm) --++(0,4.6) arc(180:90:.7cm) --++(23.4,0) arc(90:0:.7cm) 	--++(0,-.8);
		\draw[thickline] (1.5,1) arc(180:0: .6 cm and .5cm) --++(0,-1) arc(0:-90:.5cm) --++(-2.4,0) arc(270:180:.5cm) --++(0,1) arc(180:120:.5cm) arc(-60:0:.5cm) --++ 	(0,1.5);
		\draw[thick] (-.05,2) arc(180:90: 1cm) --++(2.2,0) arc(90:0:1cm) --++(0,-2);
	
		\draw[Box] (5.5,0) rectangle (7.5,1); \node at (6.5,.5) {$\Sigma g$}; \node[marked, scale=.8, above left] at (5.5,1) {};
		\draw[thickline] (6,1) arc(0:60:.5cm) arc(240:90:.5cm) --++ (2.5,0) arc (90:0:.5cm) --++(0,-3.5);
		\draw[thick] (6.5,1) arc(180:0: 1.05cm and .9cm) --++(0,-1.3) arc(0:-90:.7cm) --++(-2.75,0) arc(270:180:1cm);
		\draw[thickline] (7,1) arc(180:0: .6 cm and .5cm) --++(0,-1) arc(0:-90:.5cm) --++(-2.4,0) arc(270:180:.5cm) --++(0,1) arc(180:120:.5cm) arc(-60:0:.5cm) --++ (0,1.5);
		\draw[thick] (5.45,2) arc(180:90: 1cm) --++(2.2,0) arc(90:0:1cm) --++(0,-2) arc(180:270:1cm);
		
		\draw[thick, dotted] (10.65,-1)--++(1,0);
		
		\node at (12.1,1) {$\cdots$};

		\draw[thick, dotted] (13.55,3) --++(-1,0);
		
		\draw[Box] (16,0) rectangle (18,1); \node at (17,.5) {$\Sigma g$}; \node[marked, scale=.8, above left] at (16,1) {};
		\draw[thickline] (16.5,1) arc(0:60:.5cm) arc(240:90:.5cm) --++ (2.5,0) arc (90:0:.5cm) --++(0,-3.5);
		\draw[thick] (17,1) arc(180:0: 1.05cm and .9cm) --++(0,-1.3) arc(0:-90:.7cm) --++(-2.75,0) arc(270:180:1cm)--++(0,2) arc(0:90:1cm);
		\draw[thickline] (17.5,1) arc(180:0: .6 cm and .5cm) --++(0,-1) arc(0:-90:.5cm) --++(-2.4,0) arc(270:180:.5cm) --++(0,1) arc(180:120:.5cm) arc(-60:0:.5cm) --++ 	(0,1.5);
		\draw[thick] (15.95,2) arc(180:90: 1cm) --++(2.2,0) arc(90:0:1cm) --++(0,-2);
		
		\draw[Box] (21.5,0) rectangle (23.5,1); \node at (22.5,.5) {$\Sigma g$}; \node[marked, scale=.8, above left] at (21.5,1) {};
		\draw[thickline] (22,1) arc(0:60:.5cm) arc(240:90:.5cm) --++ (2.5,0) arc (90:0:.5cm) --++(0,-3.5);
		\draw[thick] (22.5,1) arc(180:0: 1.05cm and .9cm) --++(0,-1.3) arc(0:-90:.7cm) --++(-2.75,0) arc(270:180:1cm);
		\draw[thickline] (23,1) arc(180:0: .6 cm and .5cm) --++(0,-1) arc(0:-90:.5cm) --++(-2.4,0) arc(270:180:.5cm) --++(0,1) arc(180:120:.5cm) arc(-60:0:.5cm) --++ (0,1.5);
		\draw[thick] (21.45,2) arc(180:90: 1cm) --++(.05,0) arc(-90:0:.5cm);

		\node[right] at (25.2,0) {,};
	\end{tikzpicture}
	\end{equation*}
where again there are exactly $m$ discs containing $\Sigma g$  and for each disc we sum over the two boundary points connecting to one of the other $m-1$ discs.

Then on $\{g\in (Gr_0 \mc{P}^\Gamma)^{(R',\sigma)}_{c.s.}\colon \|g\|_{R',\sigma}\leq 1\}$,
	\begin{align*}
		F\circ \hat{c} = \hat{c} \circ T,
	\end{align*}
where
	\begin{align*}
		T=-T_1 - \frac{1}{2} T_2 + \sum_{m\geq 1}\frac{(-1)^{m+1}}{m} \left(T_{3,m} + T_{4,m}\right),
	\end{align*}
and convergence is with respect to the $\|\cdot\|_{R'}$-norm.
\end{lem}

\begin{proof}
We will prove this equivalence term by term. For $w\in Gr_0\mc{P}$ and $W=\hat{c}(w)$, we have that $\hat{c}\circ T_1(g) = W(C+\D_c\Sigma \hat{c}(g))$ immediately by Lemma \ref{composing_cyclic_derivative}. For $w\in (Gr_0\mc{P})^{(R'+1,\sigma)}_{c.s.}$, we can sum over the support of $w$ to obtain the equality since convergence is guaranteed by $\|W(C+\D_c\Sigma \hat{c}(g))\|_{R'}\leq \|W\|_{R'+1}$ (\emph{cf.} Lemma 2.5 in \cite{N13}).

Let
	\begin{equation*}
	\begin{tikzpicture}[thick, scale=.5]
		\node[left] at (-1,.5) {$\tilde{T}_2(u_1,u_2)=$};

		\draw[Box] (0,0) rectangle (2, 1); \node at (1,.5) {$u_1$}; \node[marked, scale=.8, above left] at (0, 1) {};
		\draw[thickline] (0.5,1) --++(0,1.3);
		\draw[thick] (1,1) arc(180:0: 1.1cm and .9cm) --++(0,-1) arc(180:270: 1cm) --++(2.3,0) arc(270:360:1cm);
		\draw[thickline] (1.5,1) arc(180:0: .6 cm and .5cm) --++(0,-1) arc(0:-90:.5cm) --++(-2.2,0) arc(270:180:.5cm) --++(0,2.3);

		\draw[Box] (4.3,0) rectangle (6.3,1); \node at (5.3,.5) {$u_2$}; \node[marked, scale=.8, above left] at (4.3,1) {};
		\draw[thickline] (4.8,1) --++(0,1.3);
		\draw[thick] (5.3,1) arc(180:0: 1.1cm and .9cm) --++(0,-1);
		\draw[thickline] (5.8,1) arc(180:0: .6 cm and .5cm) --++(0,-1) arc(0:-90:.5cm) --++(-2.2,0) arc(270:180:.5cm) --++(0,2.3);
		
		\node[right] at (7.5,0.3) {.};

	\end{tikzpicture}
	\end{equation*}
We will show $\hat c\circ \tilde{T}_2(u_1,u_2) = \sum_{e\in E} \sigma(e) (\D_e \hat c(u_1) )(\D_{e^\circ} \hat c(u_2))$. First assume each $u_l$, $l\in\{1,2\}$, is a delta function supported on the loop $e_{l,1}\cdots e_{l,n_l}$. Then
	\begin{align*}
		\tilde{T}_2(u_1,u_2)= \sum_{j_1=1}^{n_1}\sum_{j_2=1}^{n_2} &\delta_{e_{2,j_2}=e_{1,j_1}^\circ} \sigma(e_{2,j_2}) \sigma(e_{1,j_1+1})^2\cdots \sigma(e_{1,n_1})^2\sigma(e_{2,j_2+1})^2\cdots \sigma(e_{2,n_2})^2\\
			&\times e_{1,j_1+1}\cdots e_{1,n_1}e_{1,1}\cdots e_{1,j_1-1} e_{2,j_2+1}\cdots e_{2,n_2} e_{2,1}\cdots e_{2,j_2-1}\\
=\sum_{e\in E} \sum_{j_1=1}^{n_1} 				
			&\delta_{e_{1,j_1}=e^\circ}\sigma(e)\sigma(e^\circ)\sigma(e_{1,j_1+1})^2\cdots \sigma(e_{1,n_1})^2 e_{1,j_1+1}\cdots e_{1,n_1}e_{1,1}\cdots e_{1,j_1-1}\\
\times\sum_{j_2=1}^{n_2} &\delta_{e_{2,j_2}=e} \sigma(e) \sigma(e_{2,j_2+1})^2\cdots \sigma(e_{2,n_2})^2 e_{2,j_2+1}\cdots e_{2,n_2} e_{2,1}\cdots e_{2,j_2-1}.
	\end{align*}
Applying $\hat c$ yields
	\begin{align*}
		\hat c\circ \tilde{T}_2(u_1,u_2) = \sum_{e\in E} \sigma(e) [\D_e(C_{e_{1,1}}\cdots C_{e_{1,n_1}})][ \D_{e^\circ}(C_{e_{2,1}}\cdots C_{e_{2,n_2}})]= \sum_{e\in E} \sigma(e) (\D_e \hat c(u_1)) (\D_{e^\circ} \hat c(u_2)).
	\end{align*}
Using the multilinearity of each side we have for arbitrary $g\in Gr_0 \mc{P}^\Gamma$
	\begin{align*}
		\hat c\circ \tilde{T}_2(\Sigma g,\Sigma g) = \sum_{e\in E} \sigma(e)( \D_e \Sigma \hat c(g)) (\D_{e^\circ} \Sigma \hat c(g)),
	\end{align*}
and we note that the left-hand side is $\hat c\circ T_2(g)$.\par

Let
	
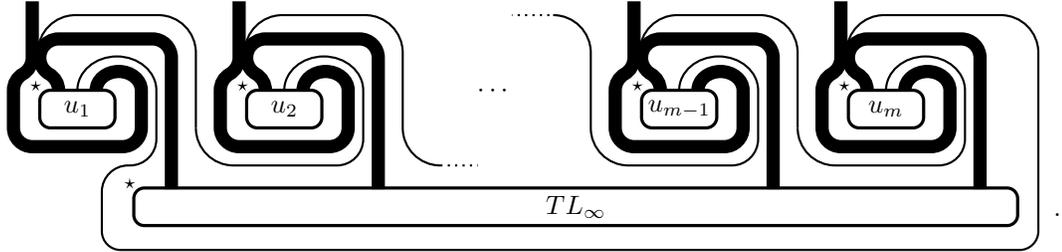
\begin{equation*}
	\begin{tikzpicture}[thick, scale=.5]
		\node[left] at (3, 4.5) {$\tilde{T}_{3,m}(u_1,\ldots,u_m)=$};

		\draw[Box] (0,0) rectangle (2,1); \node at (1,.5) {$u_1$}; \node[marked, scale=.8, above left] at (0,1) {};
		\draw[thickline] (0.5,1) arc(0:60:.5cm) arc(240:90:.5cm) --++ (2.5,0) arc (90:0:.5cm) --++(0,-3.5);
		\draw[thick] (1,1) arc(180:0: 1.05cm and .9cm) --++(0,-1.3) arc(0:-90:.7cm) --++(-.05,0) arc(90:180:.7cm) --++(0,-1.05) arc(180:270:.5cm)--++(23.9,0) arc(-90:0:.5cm);
		\draw[thickline] (1.5,1) arc(180:0: .6 cm and .5cm) --++(0,-1) arc(0:-90:.5cm) --++(-2.4,0) arc(270:180:.5cm) --++(0,1) arc(180:120:.5cm) arc(-60:0:.5cm) --++ (0,1.5);
		\draw[thick] (-.05,2) arc(180:90: 1cm) --++(2.2,0) arc(90:0:1cm) --++(0,-2);

		\draw[Box] (5.5,0) rectangle (7.5,1); \node at (6.5,.5) {$u_2$}; \node[marked, scale=.8, above left] at (5.5,1) {};
		\draw[thickline] (6,1) arc(0:60:.5cm) arc(240:90:.5cm) --++ (2.5,0) arc (90:0:.5cm) --++(0,-3.5);
		\draw[thick] (6.5,1) arc(180:0: 1.05cm and .9cm) --++(0,-1.3) arc(0:-90:.7cm) --++(-2.75,0) arc(270:180:1cm);
		\draw[thickline] (7,1) arc(180:0: .6 cm and .5cm) --++(0,-1) arc(0:-90:.5cm) --++(-2.4,0) arc(270:180:.5cm) --++(0,1) arc(180:120:.5cm) arc(-60:0:.5cm) --++ (0,1.5);
		\draw[thick] (5.45,2) arc(180:90: 1cm) --++(2.2,0) arc(90:0:1cm) --++(0,-2) arc(180:270:1cm);

		\draw[thick, dotted] (10.65,-1)--++(1,0);

		\node at (12.1,1) {$\cdots$};

		\draw[thick, dotted] (13.55,3) --++(-1,0);

		\draw[Box] (16,0) rectangle (18,1); \node at (17,.5) {$u_{m-1}$}; \node[marked, scale=.8, above left] at (16,1) {};
		\draw[thickline] (16.5,1) arc(0:60:.5cm) arc(240:90:.5cm) --++ (2.5,0) arc (90:0:.5cm) --++(0,-3.5);
		\draw[thick] (17,1) arc(180:0: 1.05cm and .9cm) --++(0,-1.3) arc(0:-90:.7cm) --++(-2.75,0) arc(270:180:1cm)--++(0,2) arc(0:90:1cm);
		\draw[thickline] (17.5,1) arc(180:0: .6 cm and .5cm) --++(0,-1) arc(0:-90:.5cm) --++(-2.4,0) arc(270:180:.5cm) --++(0,1) arc(180:120:.5cm) arc(-60:0:.5cm) --++ (0,1.5);
		\draw[thick] (15.95,2) arc(180:90: 1cm) --++(2.2,0) arc(90:0:1cm) --++(0,-2);

		\draw[Box] (21.5,0) rectangle (23.5,1); \node at (22.5,.5) {$u_m$}; \node[marked, scale=.8, above left] at (21.5,1) {};
		\draw[thickline] (22,1) arc(0:60:.5cm) arc(240:90:.5cm) --++ (2.5,0) arc (90:0:.5cm) --++(0,-3.5);
		\draw[thick] (22.5,1) arc(180:0: 1.05cm and .9cm) --++(0,-1.3) arc(0:-90:.7cm) --++(-2.75,0) arc(270:180:1cm);
		\draw[thickline] (23,1) arc(180:0: .6 cm and .5cm) --++(0,-1) arc(0:-90:.5cm) --++(-2.4,0) arc(270:180:.5cm) --++(0,1) arc(180:120:.5cm) arc(-60:0:.5cm) --++ (0,1.5);
		\draw[thick] (21.45,2) arc(180:90: 1cm) --++(3.1,0) arc(90:0:1cm) --++(0,-4.76);

		\draw[Box] (2.5,-2.6) rectangle (26,-1.6); \node at (14.25,-2.1) {$TL_\infty$}; \node[marked, scale=.8, above left] at (2.5,-1.6) {};
		
		\node[right] at (26.7,-2.3) {.};

	\end{tikzpicture}
	\end{equation*}
We claim that
\begin{align*}
\hat c\circ\tilde{T}_{3,m}&(u_1,\ldots,u_m)\\
&= (1\otimes\varphi)\circ\text{Tr}\left( \left[ U\frac{2 A^{-1}}{1+A} U^{T}\right]^{-1} \J_c\D_c \hat c(u_1) \#\J_c C\#\J_c\D_c \hat c(u_2)\#\cdots \#\J_c C\# \J_c\D_c \hat c(u_m) \right).
\end{align*}
Assume each $u_l$, $l\in\{1,\ldots, m\}$, is the delta function supported on the loop $e_{l,1}\cdots e_{l,n_l}$. Note that because of (\ref{derivative_of_C}), for each $l=1,\ldots, m-1$ and $e,f\in E$ we have
\begin{align}
\left[\J_c C \# \J_c\D_c \hat c(u_l)\right]_{ef}=& \sigma(e^\circ) [\J_c\D_c \hat c(u_l)]_{e^\circ f} = \sigma(e^\circ) \partial_f \D_{e^\circ} \hat c(u_l)\nonumber\\
=&\sigma(e^\circ) \sum_{\substack{1\leq j_l,i_l\leq n\\ j_l\neq i_l}} \sigma(e)\delta_{e_{l,j_l}=e}\sigma(f)\delta_{e_{l,i_l}=f^\circ}\left(\prod_{k=j_l+1}^{n_l} \sigma(e_{l,k})^2\right)\nonumber\\
&\times \hat c(e_{l,j_l+1}\cdots e_{l,i_l-1})\otimes \hat c(e_{l,i_l+1}\cdots e_{l,j_l-1})\nonumber\\
=& \sum_{\substack{1\leq j_l,i_l\leq n\\ j_l\neq i_l}} \delta_{e=e_{l,j_l}}\sigma(e_{l,j_l+1})^2\cdots \sigma(e_{l,n_l})^2 \sigma(e_{l,i_l}^\circ)\delta_{e_{l,i_l}=f^\circ} \label{back_terms}\\
&\times \hat c(e_{l,j_l+1}\cdots e_{l,i_l-1})\otimes \hat c(e_{l,i_l+1}\cdots e_{l,j_l-1}).\nonumber
\end{align}
Also, it follows from a simple computation (similar to \ref{linear_relation2}) that
\begin{align*}
\left(\left[ U\frac{2A^{-1}}{1+A} U^{T}\right]^{-1}\right)(e) = \left(	\begin{array}{cc}
0	&	\sigma(e^\circ)^3\\
\sigma(e)^3	&	0\end{array}\right)\qquad e\in E_+,
\end{align*}
so that
\begin{align}
\left[\left(U\frac{2A^{-1}}{1+A} U^{T}\right)^{-1}\#\J_c\D_c \hat c(u_1)\right]_{ef} =& \sigma(e^\circ)^3 [\J_c \D_c \hat c(u_1)]_{e^\circ f} = \sigma(e^\circ)^3 \partial_f \D_{e^\circ} \hat c(u_1) \nonumber\\
=& \sum_{\substack{1\leq j_1,i_1\leq n\\ j_1\neq i_1}} \sigma(e_{1,j_1}^\circ)^2 \delta_{e=e_{1,j_1}}\sigma(e_{1,j_1+1})^2\cdots \sigma(e_{1,n_1})^2 \sigma(e_{1,i_1}^\circ)\delta_{e_{1,i_1}=f^\circ} \label{front_term}\\
&\times \hat c(e_{1,j_1+1}\cdots e_{1,i_1-1})\otimes \hat c(e_{1,i_1+1}\cdots e_{1,j_1-1}).\nonumber
\end{align}
Now
\begin{align*}
\tilde{T}_{3,m}(u_1,\ldots, u_m)= \sum_{j_1=1}^{n_1} \cdots \sum_{j_m=1}^{n_m} \sum_{i_1\neq j_1,\ldots, i_m\neq j_m} &\left[\prod_{l=1}^{m-1} \sigma(e_{l,j_l+1})^2\cdots\sigma(e_{l,n_l})^2 \sigma(e_{l,i_l}^\circ) \delta_{e_{l,i_l}=e_{l+1,j_{l+1}}^\circ} \right]\\
&\times\sigma(e_{m,j_m+1})^2\cdots \sigma(e_{m,n_m})^2 \sigma(e_{m,i_m}) \delta_{e_{m,i_m}=e_{1,j_1}^\circ}\\
&\times e_{1,j_1+1}\cdots e_{1,i_1-1}\cdots e_{m,j_m+1}\cdots e_{m,i_m-1}\\
&\times \left[Tr_0(e_{m,i_m+1}\cdots e_{m,j_m-1} \cdots e_{1,i_1+1}\cdots e_{1,j_1-1})\right](s(e_{m,i_m+1})).
\end{align*}
We make the substitution $\sigma(e_{m,i_m})\delta_{e_{m,i_m}=e_{1,j_1}^\circ}=\sigma(e_{m,i_m}^\circ)\delta_{e_{m,i_m}=e_{1,j_1}^\circ}\sigma(e_{1,j_1}^\circ)^2$ , and then group the factors $\sigma(e_{1,j_1}^\circ)^2\delta_{e_{m,i_m}^\circ=e_{1,j_1}}$ with the factor corresponding to $l=1$ in the scalar product in the above equation. Also, we group the factor $\delta_{e_{l,i_l}=e_{l+1,j_{l+1}}^\circ}=\delta_{e_{l,i_l}^\circ=e_{l+1,j_{l+1}}}$ with the factor corresponding to $l+1$ rather than $l$. Finally, recall that if $u$ starts at $v$ then $[Tr_0(u)](v)=\phi_v(c(u))=\varphi(\hat{c}(u))$. With these remarks we have
	\begin{align*}
		\tilde{T}_{3,m}(u_1,\ldots, u_m)=(1\otimes [\varphi\circ\hat{c}])\left( \sum_{\substack{1\leq j_1,i_1\leq n\\ j_1\neq i_1}}\right. \sigma(e_{1,j_1}^\circ)^2 \delta_{e_{m,i_m}^\circ=e_{1,j_1}}&\sigma(e_{1,j_1+1})^2\cdots \sigma(e_{1,n_1})^2 \sigma(e_{1,i_1}^\circ)\\
			\times &(e_{1,j_1+1}\cdots e_{1,i_1-1}\otimes e_{1,i_1+1}\cdots e_{1,j_1-1}\\
\mathop{\text{\LARGE{\#}}}_{l=2}^m \left[ \sum_{\substack{1\leq j_l,i_l\leq n\\ j_l\neq i_l}}\right. \delta_{e_{l-1,i_{l-1}}^\circ=e_{l,j_l}}&\sigma(e_{l,j_l+1})^2\cdots \sigma(e_{l,n_l})^2 \sigma(e_{l,i_l}^\circ)\\
			\times&\left.\left. \vphantom{\prod_{l=2}^m \sum_{\substack{1\leq j_l,i_l\leq n\\ j_l\neq i_l}} \delta_{e=e_{l,j_l}}} e_{l,j_l+1}\cdots e_{l,i_l-1}\otimes e_{l,i_l+1}\cdots e_{l,j_l-1}\right]\right).
	\end{align*}
Applying $\hat c$ and comparing this to (\ref{back_terms}) and (\ref{front_term}) demonstrates the claimed equivalence. Then using the multilinearity of each side to replace $u_l$ with $\Sigma g$ for each $l=1,\ldots, m$ shows
\begin{align*}
\hat c\circ T_{3,m}(g)=(1\otimes \varphi)\circ\text{Tr}\left(\left[U \frac{2A^{-1}}{1+A} U^\text{T}\right]^{-1} \J_c\D_c \Sigma\hat c(g)\# \left(\J_c C \# \J_c \D_c \Sigma \hat c(g)\right)^{m-1}\right).
\end{align*}
A similar argument demonstrates
\begin{align*}
\hat c\circ T_{4,m}(g)=(\varphi\otimes 1)\circ\text{Tr}\left(\left[U\frac{2A}{1+A} U^{\text{T}}\right]^{-1} \J_c\D_c \Sigma \hat c(g) \# \left(\J_c C \# \J_c\D_c \Sigma \hat c(g)\right)^{m-1}\right).
\end{align*}
Finally, a term by term comparison then yields the equivalence $F\circ \hat c=\hat c\circ T$ on $\{g\in (Gr_0 \mc{P}^\Gamma)^{(R',\sigma)}_{c.s.}\colon \|g\|_{R',\sigma}\leq 1\}$.
\end{proof}

Using (\ref{linear_relation_derivative1}), (\ref{linear_relation3}), and (\ref{linear_relation_derivative2}) it is not hard to see that the map $F$ defined in Lemma \ref{F_is_a_tangle} is equivalent to the map considered in Corollary 3.14 of \cite{N13}. However, in the latter map $W$ is being thought of as a polynomial in the $X_e$ (for the purposes of composing with $X+\D G$).

Corollary 3.18 of \cite{N13} (with $N=|E|$) then says that there is constant $\epsilon>0$ so that if $W=\hat{c}(w)$ for $w\in (Gr_0\mc{P})^{(R'+1,\sigma)}_{c.s.}$ with $\|w\|_{R'+1,\sigma}<\epsilon$ then there exists $G\in \P^{(R',\sigma)}_{c.s.}$ so that the joint law of the $N$-tuple $Y=X+\D G$ is the free Gibbs state with potential $V_0+W$. By (\ref{circular_schwinger-dyson}) this is equivalent to joint law of the $N$-tuple $C+\D_c G$ satisfying the Schwinger-Dyson equation with potential $V_0+W$, but with the differential operators $\D_c$ and $\J_c$. That is,
\begin{align}\label{translated_S-D}
\varphi\left((C_e+\D_e G)\cdot Q(C+\D_c G)\right)=& \varphi\otimes\varphi^{op}\left([\partial_e Q](C+\D_c G)\right)\nonumber\\
&-\varphi\left([\D_e W](C+\D_c G)\cdot Q(C+\D_c G)\right),
\end{align}
where here $Q(P)$ for $Q\in \P^{(R)}$ and $P\in(\P^{(R)})^{|E|}$ means $Q$ evaluated as a power series in the $C_e$ at $C_e=P_e$.

This $G=\Sigma \hat{G}$ where $\hat{G}$ is the $\|\cdot\|_{R',\sigma}$-norm limit of the sequence $G_k=(\mathscr{S}\Pi F)^k(W)$. Thus if we define $g_k=(\mathscr{S}\Pi T)^k(w)$, then $G_k=\hat{c}(g_k)$ by Lemma \ref{F_is_a_tangle} and hence the $\|\cdot\|_{R',\sigma}$-norm limit $\hat{g}$ of the sequence $(g_k)_{k\in \N}$ satisfies $\hat{c}(\hat{g})=\hat{G}$. Let $g=\Sigma \hat{g}$. Additionally, we note that $\|g\|_{R',\sigma}$ and $\|\hat{g}\|_{R',\sigma}$ both tend to zero as $\|v-v_0\|_{R'+1,\sigma}\to 0$. This follows from Corollary 3.15(v) in \cite{N13} (specifically the last paragraph of the proof).

\begin{defi}
The element $g\in (Gr_0\mc{P})^{(R',\sigma)}_{c.s.}$ is called the \emph{transport element from $v_0$ to $v$}.
\end{defi}

Define $\eta\colon Gr_0 \mc{P} \to Gr_0[[ \mc{P}]]$ by
\[
\begin{tikzpicture}[thick, scale=.5]
\node[left] at (-1,2) {$\eta(x)=$};

\draw[Box] (0, 3) rectangle (3,4);
\node at (1.5,3.5) {$v_0 + g$};
\node[marked, scale=.8, above left] at (0,4) {};
\draw[thickline] (0.75, 4) --++(0,1.5);
\draw[thick] (1.5, 4) arc(180:0: 1.3cm and 1cm) --+(0,-1) arc (0:-90: 1.2 cm and 1cm) -- (2, 2) arc(90:180:.5cm);
\draw[thickline] (2.25,4) arc(180:0: .6cm and .5cm) --++(0,-1) arc(0:-90:.5cm) --++(-3,0) arc(-90:-180:.5cm) --++(0,2.5);

\node at (5.25,3.75) {$\cdots$};

\draw[Box] (7, 3) rectangle (10,4);
\node at (8.5,3.5) {$v_0 + g$};
\node[marked, scale=.8, above left] at (7,4) {};
\draw[thickline] (7.75, 4) --++(0,1.5);
\draw[thick] (8.5, 4) arc(180:0: 1.3cm and 1cm) --++(0,-1) arc (0:-90: 1.2 cm and 1cm) -- (8.75, 2) arc(90:180:.5cm);
\draw[thickline] (9.25,4) arc(180:0: .6cm and .5cm) --++(0,-1) arc(0:-90:.5cm) --++(-3,0) arc(-90:-180:.5cm) --++(0,2.5);

\draw[Box] (1,.5) rectangle (9,1.5);
\node at (5,1) {$x$};
\node[marked,scale=.8, above left] at (1,1.5) {};
\node[below right] at (9,1) {,};
\end{tikzpicture}
\]
where the number of discs containing  $v_0 + g$ varies according to the components of $x$ and for each such disc we sum over the boundary points connecting to $x$. From Lemma \ref{composing_cyclic_derivative} it follows that $\hat{c}\circ\eta(x)=[\hat{c}(x)](C+\D_c G)$.

Moreover, we claim $\eta(x)\in (Gr_0\mc{P})^{(R)}$ for each $x\in Gr_0\mc{P}$. Fix $x\in Gr_0\mc{P}$. Since $g\in (Gr_0\mc{P})^{(R')}$, there is a sequence $\{h_n\}_{n\in \N}\subset Gr_0\mc{P}$ so that $\|g-h_n\|_{R'}\to 0$. Let
\[
\begin{tikzpicture}[thick, scale=.5]
\node[left] at (-1,2) {$x_n=$};

\draw[Box] (0, 3) rectangle (3,4);
\node at (1.5,3.5) {$v_0 + h_n$};
\node[marked, scale=.8, above left] at (0,4) {};
\draw[thickline] (0.75, 4) --++(0,1.5);
\draw[thick] (1.5, 4) arc(180:0: 1.3cm and 1cm) --+(0,-1) arc (0:-90: 1.2 cm and 1cm) -- (2, 2) arc(90:180:.5cm);
\draw[thickline] (2.25,4) arc(180:0: .6cm and .5cm) --++(0,-1) arc(0:-90:.5cm) --++(-3,0) arc(-90:-180:.5cm) --++(0,2.5);

\node at (5.25,3.75) {$\cdots$};

\draw[Box] (7, 3) rectangle (10,4);
\node at (8.5,3.5) {$v_0 + h_n$};
\node[marked, scale=.8, above left] at (7,4) {};
\draw[thickline] (7.75, 4) --++(0,1.5);
\draw[thick] (8.5, 4) arc(180:0: 1.3cm and 1cm) --++(0,-1) arc (0:-90: 1.2 cm and 1cm) -- (8.75, 2) arc(90:180:.5cm);
\draw[thickline] (9.25,4) arc(180:0: .6cm and .5cm) --++(0,-1) arc(0:-90:.5cm) --++(-3,0) arc(-90:-180:.5cm) --++(0,2.5);

\draw[Box] (1,.5) rectangle (9,1.5);
\node at (5,1) {$x$};
\node[marked,scale=.8, above left] at (1,1.5) {};
\node[below right] at (9,1) {,};
\end{tikzpicture}
\]
then $x_n\in Gr_0\mc{P}$ and $\eta(x)$ is the $\|\cdot\|_R$-limit of the $x_n$ by Lemma 2.5 in \cite{N13}.

It is clear that the element associated to $\varphi\circ\hat{c}\circ\eta$ via the duality in (\ref{sum_pairing}) is
\[
\begin{tikzpicture}[thick, scale=.5]
\node[left] at (-1,3.5) {$TL_\infty^{(v)}=$};
\draw[Box] (-1,5.5) rectangle (8.25,6.5);
\node at (3.625,6) {$TL_\infty$};
\node[marked, scale=.8, above left] at (-1,6.5) {};

\draw[Box] (0, 3) rectangle (3,4);
\node at (1.5,3.5) {$v_0 + g$};
\node[marked, scale=.8, above left] at (0,4) {};
\draw[thickline] (0.75, 4) --++(0,1.5);
\draw[thick] (1.5, 4) arc(180:0: 1.3cm and 1cm) --+(0,-1) arc (0:-90: 1.2 cm and 1cm) -- (2, 2) arc(90:180:.5cm);
\draw[thickline] (2.25,4) arc(180:0: .6cm and .5cm) --++(0,-1) arc(0:-90:.5cm) --++(-3,0) arc(-90:-180:.5cm) --++(0,2.5);

\node at (5.25,3.75) {$\cdots$};

\draw[Box] (7, 3) rectangle (10,4);
\node at (8.5,3.5) {$v_0 + g$};
\node[marked, scale=.8, above left] at (7,4) {};
\draw[thickline] (7.75, 4) --++(0,1.5);
\draw[thick] (8.5, 4) arc(180:0: 1.3cm and 1cm) --++(0,-1) arc (0:-90: 1.2 cm and 1cm) -- (8.75, 2) arc(90:180:.5cm);
\draw[thickline] (9.25,4) arc(180:0: .6cm and .5cm) --++(0,-1) arc(0:-90:.5cm) --++(-3,0) arc(-90:-180:.5cm) --++(0,2.5);
\node[right] at (11,3.5) {$\in Gr_0[[\mc{P}]]$,};

\end{tikzpicture}
\]
where we sum over the number of input discs containing $v_+ g$, and for each disc we sum over the boundary point connected to the bottom of the diagram. Define
\[
Tr_0^{(v)}(x):=\<TL_\infty^{(v)},x\>_{\mc{P}}
\]
(we note $TL_\infty^{(v)}={TL_\infty^{(v)}}^*$ since $v_0$, $g$, and $TL_\infty$ are all self-adjoint), then $Tr_0^{(v)}=Tr_0\circ\eta$.

The above observations and Lemma \ref{S-D_planar_tangle_lemma} immediately imply the following proposition.

\begin{prop}\label{solving_S-D_planar_tangle}
There exists $\epsilon>0$ such that when $v\in (Gr_0\mc{P})^{(R'+1,\sigma)}_{c.s.}$ satisfies $\|v-v_0\|_{R'+1,\sigma}<\epsilon$, there is $g\in (Gr_0\mc{P})^{(R',\sigma)}_{c.s.}$ so that $TL_\infty^{(v)}\in Gr_0[[\mc{P}]]$ defined as above satisfies the Schwinger-Dyson planar tangle.

Moreover, the map $\hat{c}\circ\eta$ sends $Gr_0\mc{P}^\Gamma$ to a subalgebra of $W^*(C_e+\D_e\hat{c}(g)\colon e\in E)$. The joint law of the generators $\{C_e + \D_e\hat{c}(g)\}_{e\in E}$ with respect to the free quasi-free state $\varphi$ is the free Gibbs state with potential $[\hat{c}(v)](C+\D_c\hat{c}(g)) = \hat{c}\circ\eta(v)$.
\end{prop}

\begin{rem}
The Schwinger-Dyson planar tangle on $Gr_0 \mc{P}$ was solved in Proposition 2 of \cite{GJSZJ12} for potentials of the form $v_0+\sum_{i=1}^k t_i B_i$, $B_1,\ldots, B_k\in Gr_0 \mc{P}$ with $\sum_i |t_i|$ small. Proposition \ref{solving_S-D_planar_tangle} extends this to $B_1,\ldots, B_k\in (Gr_0\mc{P})^{(R'+1,\sigma)}$ despite its requirement that $B_1,\ldots, B_k$ are invariant under $\rho$.

Indeed, let $v =v_0+ \sum_{i=1}^k t_i B_i$, with $B_1,\ldots , B_k\in (Gr_0\mc{P})^{(R'+1,\sigma)}$ and $\sum_{i=1}^k |t_i|$ small. Since elements of $(Gr_0\mc{P})^{(R'+1,\sigma)}$ are automatically invariant under $\sigma_{-i}^\varphi$ (simply because the planar tangle
\[
\begin{tikzpicture}[thick, scale=.5]
\draw[Box] (0,0) rectangle (1.5,1);
\draw[thickline] (.75,1) arc(180:0: .6cm and .5cm) --++(0,-1) arc(0:-90:.5cm) --++(-1.5,0) arc(-90:-180:.5cm) --++(0,1.6);

\end{tikzpicture}
\]
is isotopically equivalent to the identity planar tangle), $\tilde{v}:=\mathscr{S}(v)\in (Gr_0\mc{P})^{(R'+1,\sigma)}_{c.s.}$ is invariant under $\rho$. So we apply Proposition \ref{solving_S-D_planar_tangle} to $\tilde{v}$ to obtain $TL_\infty^{(\tilde{v})}\in Gr_0[[\mc{P}]]$ satisfying the Schwinger-Dyson planar tangle with potential $\tilde{v}$. But then Lemma 2.6 in \cite{N13} implies $\D_c\mathscr{S}\hat{c}(v-v_0)=\D_c\hat{c}(v-v_0)$. So using Lemma \ref{composing_cyclic_derivative} to translate this to planar tangles we see that we can simply replace $\tilde{v}$ with $v$ in the Schwinger-Dyson planar tangle, and hence $TL_\infty^{(\tilde{v})}$ also satisfies the Schwinger-Dyson planar tangle with potential $v$.
\end{rem}

\subsection{Equality of non-commutative probability spaces}

Using $\hat{c}$ to realize $Gr_0\mc{P}$ as a subalgebra of $M$, we let $M_0=W^*(\hat{c}(Gr_0\mc{P}))\subset M$ and $M_{0,\pm}=W^*(\hat{c}(Gr_0^\pm \mc{P}))$. Note that by our choice of $R\geq 4\delta^\frac{1}{2}$, $\|\cdot\|_S$ dominates the operator norm for any $S\geq R$ and therefore $(Gr_0\mc{P})^{(S)}\subset M_0$ for every $S\geq R$. Thus, for $v\in (Gr_0\mc{P})^{(R'+1,\sigma)}_{c.s.}$ with $\|v-v_0\|_{R',\sigma}<\epsilon$ ($\epsilon$ as in Proposition \ref{solving_S-D_planar_tangle}) we have $\eta(x) \in (Gr_0\mc{P})^{(R)}\subset M_0$ for each $x\in Gr_0\mc{P}$. Consider $M_0^{(v)}=W^*(\hat{c}\circ\eta(Gr_0\mc{P}))\subset M_0$ and $M_{0,\pm}^{(v)} =W^*(\hat{c}\circ\eta(Gr_0^\pm\mc{P}))$. In this section we show that by making $\epsilon$ smaller if necessary we have $M_0=M_0^{(v)}$.

\begin{lem}\label{loopify}
Let $R>0$. For $H\in \left(\P^{(R)}\right)^{|E|}$, define
	\begin{align*}
		L(H):= (\J_c C\#H)\# C = \sum_{e\in E} \sigma(e) H_e C_{e^\circ}.
	\end{align*}
Suppose $h\in (Gr_0 \mc{P})^{(R)}$ with zero $\mc{P}_0$ component embeds into $M$ as
	\begin{align*}
		\hat{c}(h)=\sum_{e\in E}\sum_{ue\in L} \beta_h(ue) \hat{c}(u)C_e,
	\end{align*}
and define $H\in \left(\P^{(R)}\right)^{|E|}$ by
	\begin{align*}
		H_e = \sum_{ue^\circ\in L} \sigma(e^\circ)\beta_h(ue^\circ) \hat{c}(u).
	\end{align*}
Then $L(H) = \hat{c}(h)$ and for $u_1 e u_2\in L$ ($e\in E$) we have
\[
\begin{tikzpicture}[thick, scale=.5]
\draw[Box] (0,0) rectangle (5,1);
\node at (2.5,0.5) {$e$};
\node at (1,0.5) {$u_1$};
\node at (4,0.5) {$u_2$};
\node[marked, scale=.8, above left] at (0,1) {};
\draw[thickline] (1,1) --++ (0,2);
\draw[thickline] (4,1) --++ (0,2);
\draw[Box] (2,1.5) rectangle (3,2.5);
\node at (2.5,2) {$h$};
\node[marked, scale=.8, above left] at (2,2.5) {};
\draw[thickline] (2.4,2.5) --++(0,0.5);
\draw[thick] (2.75,2.5) arc(180:0:0.25 cm) --++ (0,-1) arc(360:270:0.25) --++ (-0.25,0) arc(90:180:0.25);
\node[right] at (5, 1.5) {$\stackrel{\hat c}{\longmapsto}$};
\node[right] at (6.5, 1.5) {$\hat c(u_1) H_e \hat c(u_2).$};
\end{tikzpicture}
\]
\end{lem}
\begin{proof}
The assertion $L(H)=\hat c(h)$ follows immediately from the definition of $H$ and $L(H)$. To see that the output of the planar tangle embeds as stated, one simply notes that the string connecting $h$ to $e$ must have $e^\circ$ as its endpoint in $h$ and contributes a factor of $\sigma(e^{\circ})$ to the tangle.
\end{proof}

\begin{thm}\label{vNas_are_equal}
There exists a constant $\epsilon>0$ so that for $v\in (Gr_0 \mc{P})_{c.s.}^{(R'+1,\sigma)}$ with $\|v-v_0\|_{R'+1,\sigma}<\epsilon$, $M_0=M_0^{(v)}$. Moreover, there exists a $*$-automorphism of $M$ which fixes $M_0$ and takes the free Gibbs state with potential $\hat{c}(v_0)$ to the free Gibbs state with potential $\hat{c}(v)$.
\end{thm}
\begin{proof}
The inclusion $M_0^{(v)}\subset M_0$ was already demonstrated at the beginning of this section. Towards showing the reverse inclusion, fix $x\in Gr_0\mc{P}$ and consider the following recursively defined sequence: $h_0=v_0$ and
\[	
	\begin{tikzpicture}[thick, scale=.5]
		\node[left] at (0,1) {$h_{k+1}= v_0 -$};
		\draw[Box] (0,0.5) rectangle (5,1.5);
		\node at (2.5,1) {$\hat{g}$};
		\node[marked, scale=.8, above left] at (0,1.5) {};
		\draw[thick] (4.6,1.5) --++ (0,2.5);
		
		\draw[Box] (0,2) rectangle (1,3);
		\node at (0.5,2.5) {$h_k$};
		\node[marked, scale=.8, above left] at (0,3) {};
		\draw[thickline] (0.4,3) --++(0,1);
		\draw[thick] (0.75,3) arc(180:0:0.25 cm) --++ (0,-1) arc(360:270:0.25) --++ (-0.25,0) arc(90:180:0.25);
		
		\node at (2.2,2) {$\cdots$};

		\draw[Box] (3,2) rectangle (4,3);
		\node at (3.5,2.5) {$h_k$};
		\node[marked, scale=.8, above left] at (3,3) {};
		\draw[thickline] (3.4,3) --++(0,1);
		\draw[thick] (3.75,3) arc(180:0:0.25 cm) --++ (0,-1) arc(360:270:0.25) --++ (-0.25,0) arc(90:180:0.25);
		
		\node[right] at (5,0.7) {,};
	\end{tikzpicture}
\]
where $\hat{g}=\mathscr{N}g \in (Gr_0\mc{P})^{(R',\sigma)}_{c.s.}$ with $g$ the transport element from $v_0$ to $v$. Letting $R''=\max\{R,\|Y\|_R\}$ ($Y=X+\D\hat{c}(g)$ as in the discussion following Lemma \ref{F_is_a_tangle}), we claim that $h_k\in (Gr_0\mc{P})^{(R'')}$ and if
\[
\begin{tikzpicture}[thick, scale=.5]
\node[left] at (0,1) {$x_k=$};
\draw[Box] (0,0.5) rectangle (4.5,1.5);
\node at (2.25,1) {$x$};
\node[marked, scale=.8, above left] at (0,1.5) {};
\draw[Box] (0,2) rectangle (1,3);
\node at (0.5,2.5) {$h_k$};
\node[marked, scale=.8, above left] at (0,3) {};
\draw[thickline] (0.4,3) --++(0,1);
\draw[thick] (0.75,3) arc(180:0:0.25 cm) --++ (0,-1) arc(360:270:0.25) --++ (-0.25,0) arc(90:180:0.25);
\node at (2.2,2) {$\cdots$};
\draw[Box] (3,2) rectangle (4,3);
\node at (3.5,2.5) {$h_k$};
\node[marked, scale=.8, above left] at (3,3) {};
\draw[thickline] (3.4,3) --++(0,1);
\draw[thick] (3.75,3) arc(180:0:0.25 cm) --++ (0,-1) arc(360:270:0.25) --++ (-0.25,0) arc(90:180:0.25);
\node[right] at (4.5,0.8) {,};
\end{tikzpicture}
\]
then $x_k\in (Gr_0\mc{P})^{(R'')}$, $\eta(x_k)\in (Gr_0\mc{P})^{(R)}$, and $\eta(x_k)\to x$ in the $\|\cdot\|_{R}$-norm.

Indeed, suppose $\hat{G}=\hat{c}(\hat{g}) = \sum_{e\in E}\sum_{ue\in L}\beta_{\hat{g}}(ue) \hat{c}(u) C_e$. From Lemma 2.5 in \cite{N13} it follows that if $f=\D_c \hat{c}(g) = \D_c \Sigma \hat{G}$, then $f_e = \sum_{ue^\circ\in L} \sigma(e^\circ)\beta_{\hat{g}}(ue^\circ) \hat{c}(u)$ and $f\in (\P^{(R')})^{|E|}$. We then see by Lemma \ref{loopify} that $L(f)=\hat{G}=\hat{c}(\hat{g})$. 

Note that
	\begin{align*}
		Y=X+ \D \hat{c}(g) = X+U^{-1}\# \D_c\hat{c}(g)) = X + U^{-1}\# f.
	\end{align*}
We also have for $S\leq R'$
	\begin{align*}
		\|f \|_S \leq  \delta^\frac{1}{2} \|\hat{g}\|_S \leq \delta^\frac{1}{2} \|\hat{g}\|_{R',\sigma},
	\end{align*}
which tends to zero as $\|v-v_0\|_{R'+1,\sigma}\to 0$. So by taking $\epsilon$ sufficiently small we have
	\begin{align*}
		\| Y\|_R \leq R + \|U^{-1}\#f\|_R < R'.
	\end{align*}
We will need this shortly when we appeal to Lemma 2.8 from \cite{N13} because it implies $R''=\max\{R,\|Y\|_R\}<R'$.

For each $k$, define an $|E|$-tuple $H_k$ of (\textit{a priori} formal) power series in the $C_e$ so that $L(H_k)=\hat{c}(h_k)$. In particular, $H_0 = C$ since $L(C)=\hat{c}(v_0)$. Then these $H_k$ satisfy the recursive relationship $H_{k+1}=C-f(H_k)$ since by Lemma \ref{loopify},
\begin{align*}
L(H_{k+1})=\hat{c}(h_{k+1}) &= V_0 - \sum_{e\in E} \sum_{e_1\cdots e_r e\in L} \beta_{\hat{g}}(e_1\cdots e_r e) [H_k]_{e_1}\cdots [H_k]_{e_r} C_e \\
&= V_0 - \sum_{e\in E} \sigma(e^\circ) [f(H_k)]_{e^\circ} C_e = L(C - f(H_k)),
\end{align*}
and the map $L$ is injective.

The sequence $\{U^{-1}\# H_k\}_{k\in\N}$ (now thought of as $|E|$-tuples of power series in the $X_e$), is precisely the sequence considered in Lemma 2.8 of \cite{N13} for $S=R'$ and $f$ replaced with $U^{-1}\# f$. We saw above that $\|U^{-1}\# f\|_{R'}$ can be made arbitrarily small by shrinking $\epsilon$, so let $\epsilon$ be small enough that $\|U^{-1}\# f\|_{R'}<C$ for $C$ as in Lemma 2.8 of \cite{N13}. By replacing $\|Y\|_R$ with $R''$ in the proof of Lemma 2.8, we see that in addition to obtaining $U^{-1}\# H_k(Y)\in \left(\P^{(R)}\right)^{|E|}$ and $U^{-1}\# H_k(Y)\to X$  (with respect to the $\|\cdot\|_{R}$-norm and evaluating $U^{-1}\# H_k$ in the $X_e$), we also have $U^{-1}\# H_k\in (\P^{(R'')})^{|E|}$. Consequently, $H_k(C+\D_c\hat{c}(g))\in \left(\P^{(R)}\right)^{|E|}$, $H_k(C+\D_c\hat{c}(g))\rightarrow C$ (with respect to the $\|\cdot\|_{R}$-norm and evaluating $H_k$ in the $C_e$), and $H_k\in (\P^{(R'')})^{|E|}$.

Now,
\begin{align*}
\|h_k\|_{R''} = \|L(H_k)\|_R \leq |E|\sqrt{\delta}\max_{e\in E} \|[H_k]_eC_{e^\circ}\|_{R''} <\infty;
\end{align*}
that is, $h_k\in (Gr_0\mc{P})^{(R'')}$. Next, if $x$ embeds as $\sum_{u\in L} \beta_x(u) u\in Gr_0\mc{P}^\Gamma$, then by Lemma \ref{loopify}
\[
\hat{c}(x_k)=\sum_{e_1\cdots e_r\in L}\beta_x(e_1\cdots e_r)[H_k]_{e_1}\cdots [H_k]_{e_r},
\]
which implies $x_k\in (Gr_0\mc{P})^{(R'')}$ since $\P^{(R'')}$ is a Banach algebra. Furthermore,
\begin{equation}
\hat{c}\circ \eta(x_k)=\sum_{e_1\cdots e_r\in L}\beta_x(e_1\cdots e_r)[H_k(C+\D_c\hat{c}(g))]_{e_1}\cdots [H_k(C+\D_c\hat{c}(g))]_{e_r},
\end{equation}
which implies $\eta(x_k)\in (Gr_0\mc{P})^{(R)}$ as claimed. Additionally, since $H_k(C+\D_c\hat{c}(g))\to C$ we have
\begin{align*}
\hat{c}\circ \eta(x_k)\to \sum_{e_1\cdots e_r\in L}\beta_x(e_1\cdots e_r)C_{e_1}\cdots C_{e_r} =\hat{c}(x)
\end{align*}
in the $\|\cdot\|_R$-norm, which implies $\eta(x_k)\to x$ in the $\|\cdot\|_R$-norm.

Now, let $\pi_n\colon Gr_0[[\mc{P}]]\to \mc{P}_n$ be the projection onto the $n$th component. For each $k$ and $N$, write $x_k^N=\sum_{n=0}^N \pi_n(x_k)$. Then $\lim_N \|x_k - x_k^N\|_{R''}=0$ for each $k$. Hence
	\begin{align*}
		\lim_{N\to\infty}\| \eta(x_k) - \eta(x_k^N) \|_R &= \lim_{N\to\infty}\| \underbrace{\left[\hat{c}(x_k - x_k^N)\right]}_{\text{as a polynomial in the $C_e$}}(C+\D_c\hat{c}(g))\|_R \\
			&= \lim_{N\to\infty}\| \underbrace{\left[\hat{c}(x_k - x_k^N)\right]}_{\text{as a polynomial in the $X_e$}}(Y)\|_R\\
			&\leq \lim_{N\to\infty}\| \hat{c}(x_k - x_k^N)\|_{R''} = \lim_{N\to\infty}\| x_k - x_k^N \|_{R''}= 0.
	\end{align*}
This shows that $\eta(x_k) \in \overline{\hat{c}\circ\eta( Gr_0\mc{P})}^{\|\cdot\|_{R}}\subset  M_0^{(v)}$ and hence $x\in M_0^{(v)}$ as the $\|\cdot\|_{R}$-limit of the $\eta(x_k)$.

Finally, the $*$-automorphism on $M$ is simply the extension of $C_e\mapsto C_e+\D_e\hat{c}(g)$.
\end{proof}

\begin{rem}
Because of (\ref{trace_relation}), the embedding $\hat{c}\colon (Gr_0\mc{P},Tr_0)\hookrightarrow (M_0,\varphi)$ is not trace-preserving. However, restricting to either $Gr_0^+\mc{P}$ and $Gr_0^-\mc{P}$ and normalizing $\hat{c}$ by $\frac{1}{|V_\pm|}$ does yield a trace-preserving embedding. Similarly, $\frac{1}{|V_\pm|}\hat{c}\circ\eta$ is a trace-preserving embedding of $(Gr_0\mc{P},Tr_0^{(v)})\hookrightarrow (M_{0,\pm}^{(v)},\varphi)$.

Since it is clear that Theorem \ref{vNas_are_equal} also gives the equalities $M_{0,\pm}=M_{0,\pm}^{(v)}$, we observe that $\frac{1}{|V_\pm|}\hat{c}$ and $\frac{1}{|V_\pm|}\hat{c}\circ\eta$ are distinct embeddings of $Gr_0^\pm\mc{P}$ into $\B(\mc{F})$ which generate the same von Neumann algebra.
\end{rem}

\begin{rem}
Since the proof Theorem \ref{vNas_are_equal} relied only on operator norm convergence, the result also holds when the von Neumann algebras are replaced with the $C^*$-algebras.
\end{rem}

%%%%%%%%%%%%%%%%%%%%%%%%%%%%%%%%%%%%%%%%%%%%%%%%%%%%%%%%
%	Higher order von Neumann algebras

\subsection{Tower of non-commutative probability spaces}

In this section we recall the embeddings of $Gr_k\mc{P}^\Gamma$ into $\B(\mc{F})$ considered in \cite{GJS10}, and show that perturbing these embeddings by the transport element $g$ still yields the same von Neumann algebra.

For $k\geq 1$ consider the map $\hat{c}_k\colon Gr_k\mc{P}^\Gamma\to \B(\mc{F})$ defined by
\begin{align*}
\hat{c}_k(uf_k^\circ\cdots f_1^\circ e_1\cdots e_k) = \hat{\ell}(e_1)\cdots \hat{\ell}(e_k)\hat{c}(u)\hat{\ell}(f_k)^*\cdots \hat{\ell}(f_1)^*,
\end{align*}
where $e_1,\ldots, e_k,f_1,\ldots, f_k\in E$ and $uf_k^\circ\cdots f_1^\circ e_1\cdots e_k \in L$. We let $\hat{c}_0=\hat{c}$. The reason for the apparent rotation of the edges in the definition of $\hat{c}_k$ is that when we represent $x\in Gr_k\mc{P}$ as the diagram
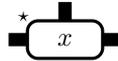
\[
\begin{tikzpicture}[thick, scale=.5]
\draw[Box] (0, 0) rectangle (2,1);
\node at (1,0.5) {$x$};
\node[marked, scale=.8, above left] at (0,1) {};
\draw[thickline] (0, 0.5) --++ (-0.5,0);
\draw[thickline] (2, 0.5) --++ (0.5,0);
\draw[thickline] (1,1) --++ (0,0.5);
\end{tikzpicture}
\]
we want to send the strings on the left to operators of the form $l(e)$, the strings on the right to operators of the form $l(e^\circ)^*$, and the strings on top to operators of the form $\hat{c}(e)$. Because $l(f)^*l(e)=\delta_{f=e}\|e\|^2$, $\hat{c}_k$ is a $*$-homomorphism from $Gr_k\mc{P}^\Gamma$ (with multiplication $\wedge_k$) to $\mathfrak{M}_k\subset \B(\mc{F})$ where
\[
\mathfrak{M}_k=\text{span}\{ \hat{\ell}(e_1)\cdots \hat{\ell}(e_k)\hat{c}(u)\hat{\ell}(f_k)^*\cdots \hat{\ell}(f_1)^*\colon e_1\cdots e_k u f_k^\circ\cdots f_1^\circ \in L\}.
\]

Also considered in \cite{GJS10} was the trace $\varphi_k\colon \mathfrak{M}_k\to \C$ defined by
\[
\varphi_k(\cdot) = \delta^{-k}\sum_{f_1,\ldots,f_k\in E} \left(\frac{\mu(s(f_1))}{\mu(t(f_k))}\right)^{\frac{1}{2}} \<f_1\otimes\cdots \otimes f_k,\ \cdot\  f_1\otimes\cdots\otimes f_k\>_{\mc{F}},
\]
which satisfies $\varphi_k(\hat{c}(x))= \sum_{v\in V} [Tr_k(x)](v)$ for $x\in Gr_k{\mc{P}}$, and the embeddings $i^{k-1}_k\colon \mathfrak{M}_{k-1}\to \mathfrak{M}_k$ defined by
\[
i^{k-1}_k( \hat{c}_{k-1}(u) ) = \sum_{eue^\circ\in L} \sigma(e)^{-1} \hat{\ell}(e)\hat{c}_{k-1}(u)\hat{\ell}(e)^*,
\]
so that $\varphi_k\circ i^{k-1}_k = \varphi_{k-1}$.

These inclusion maps correspond to the inclusion tangles $I^{k-1}_k\colon Gr_{k-1}\mc{P}\to Gr_k\mc{P}$ defined by
\[
\begin{tikzpicture}[thick, scale=.5]
\node[left] at (-0.5,0.5) {$I^{k-1}_k(x)=$};
\draw[Box] (0, 0) rectangle (2,1);
\node at (1,0.5) {$x$};
\node[marked, scale=.8, above left] at (0,1) {};
\draw[thickline] (0, 0.5) --++ (-0.5,0);
\draw[thickline] (2, 0.5) --++ (0.5,0);
\draw[thickline] (1,1) --++ (0,0.5);
\draw[thick] (-0.5,-0.255) --++ (3,0);
\node[right] at (2.5,0) {,};
\end{tikzpicture}
\]
in the sense that $i^{k-1}_k\circ\hat{c}_{k-1} = \hat{c}_k\circ I^{k-1}_k$.

For each $k\geq 0$, let $M_k=W^*(\hat{c}_k(Gr_k\mc{P})) \subset \mathfrak{M}_k$ and $M_{k,\pm}=W^*(\hat{c}_k(Gr_k^\pm\mc{P}))$. In \cite{GJS10}, the embedding $c\colon \C\<e\in E\>\to \B(\mc{F}_A)$ rather than $\hat{c}$ was used to define these von Neumann algebras on the GNS space corresponding to the weight $\phi$ from Section \ref{GJS_construction}. However, since $\varphi\circ\hat{c}=\phi\circ c$ these are isomorphic to the $M_k$ defined here. Consequently, Theorem 8 of \cite{GJS10} implies that the standard invariant of the subfactors $i^{k-1}_k(M_{k-1,+})\subset M_{k,+}$ is isomorphic to the subfactor planar algebra $\mc{P}$.

Let $v\in (Gr_0\mc{P})^{(R'+1,\sigma)}_{c.s.}$ be sufficiently close to $v_0$ so that the transport element $g$ from $v_0$ to $v$ exists. We then define $\eta_k$ on $Gr_k\mc{P}$ by
\[
	\begin{tikzpicture}[thick, scale=.5]
		\node[left] at (-1,2) {$\eta_k(x)=$};

		\draw[Box] (0, 3) rectangle (3,4);
		\node at (1.5,3.5) {$v_0 + g$};
		\node[marked, scale=.8, above left] at (0,4) {};
		\draw[thickline] (0.75, 4) --++(0,1.5);
		\draw[thick] (1.5, 4) arc(180:0: 1.3cm and 1cm) --+(0,-1) arc (0:-90: 1.2 cm and 1cm) -- (2, 2) arc(90:180:.5cm);
		\draw[thickline] (2.25,4) arc(180:0: .6cm and .5cm) --++(0,-1) arc(0:-90:.5cm) --++(-3,0) arc(-90:-180:.5cm) --++(0,2.5);

		\node at (5.25,3.75) {$\cdots$};
	
		\draw[Box] (7, 3) rectangle (10,4);
		\node at (8.5,3.5) {$v_0 + g$};
		\node[marked, scale=.8, above left] at (7,4) {};
		\draw[thickline] (7.75, 4) --++(0,1.5);
		\draw[thick] (8.5, 4) arc(180:0: 1.3cm and 1cm) --++(0,-1) arc (0:-90: 1.2 cm and 1cm) -- (8.75, 2) arc(90:180:.5cm);
		\draw[thickline] (9.25,4) arc(180:0: .6cm and .5cm) --++(0,-1) arc(0:-90:.5cm) --++(-3,0) arc(-90:-180:.5cm) --++(0,2.5);

		\draw[Box] (1,.5) rectangle (9,1.5);
		\node at (5,1) {$x$};
		\node[marked,scale=.8, above left] at (1,1.5) {};
		\draw[thickline] (1,1) --++ (-1,0);
		\draw[thickline] (9,1) --++ (1,0);
		\node at (15, 2) {$x\in Gr_k\mc{P}$};
	\end{tikzpicture}
\]
(with $\eta_0=\eta$). Note that $\eta_k(x\wedge_k y)= \eta_k(x)\wedge_k \eta_k(y)$ for $x,y\in Gr_k\mc{P}$ and that $I^{k-1}_k$ intertwines $\eta_{k-1}$ and $\eta_k$.

\begin{thm}\label{k_vNas_are_equal}
Let $\epsilon>0$ be as in Theorem \ref{vNas_are_equal}. For each $x\in Gr_k\mc{P}$, $\hat{c}_k\circ\eta_k(x)\in M_k$. Moreover, if $M_k^{(v)}=W^*(\hat{c}_k\circ\eta_k(Gr_k\mc{P}))$ then $M_k^{(v)}=M_k$ and $M_{k,\pm}^{(v)}=M_{k,\pm}$. Finally, the inclusions in the tower $\{M_k^{(v)}\}_{k\geq 0}$ given by the maps $\{i_k^{k-1}\}_{k\geq 0}$ are the same as in the tower $\{M_k\}_{k\geq 0}$; that is, $\mc{P}$ is recovered as the standard invariant of the tower $\{M_{k,+}^{(v)}\}_{k\geq 0}$.
\end{thm}
\begin{proof}
Let $\{h_n\}_{n\geq 0}\subset Gr_0\mc{P}$ be a sequence converging to $v_0+g$ with respect to the $\|\cdot\|_{R',\sigma}$-norm. Given $x\in Gr_k\mc{P}$, suppose it embeds as $\sum_{uu_2^\circ u_1\in L} \beta_x(uu_2^\circ u_1) uu_2^\circ u_1$ where $u_1,u_2$ are paths of length $k$. Define $\hat{\ell}(e_1\cdots e_k):=\hat{\ell}(e_1)\cdots \hat{\ell}(e_k)$ and $\|e_1\cdots e_k\|:=\|e_1\|\cdots \|e_k\|$. Then
	\begin{align*}
		\hat{c}_k\circ\eta_k(x) = \sum_{uu_2^\circ u_1\in L} \beta_x(uu_2^\circ u_1) \hat{\ell}(u_1) \hat{c}\circ\eta(u) \hat{\ell}(u_2)^*.
	\end{align*}
Now, $\hat{c}\circ\eta(u)=[\hat{c}(u)](C+\D_c\hat{g})$ is the $\|\cdot\|_R$-norm limit (and hence operator norm limit) of $[\hat{c}(u)](\D_c\hat{c}(h_n))$. Also
	\begin{align*}
		\sum_{uu_2^\circ u_1\in L} \beta_x(uu_2^\circ u_1) \hat{\ell}(u_1)[\hat{c}(u)](\D_c\hat{c}(h_n)) \hat{\ell}(u_2)^* = \hat{c}_k(x_n),
	\end{align*}
where
\[
\begin{tikzpicture}[thick, scale=.5]
\node[left] at (-1,2) {$x_n=$};

\draw[Box] (0, 3) rectangle (3,4);
\node at (1.5,3.5) {$h_n$};
\node[marked, scale=.8, above left] at (0,4) {};
\draw[thickline] (0.75, 4) --++(0,1.5);
\draw[thick] (1.5, 4) arc(180:0: 1.3cm and 1cm) --+(0,-1) arc (0:-90: 1.2 cm and 1cm) -- (2, 2) arc(90:180:.5cm);
\draw[thickline] (2.25,4) arc(180:0: .6cm and .5cm) --++(0,-1) arc(0:-90:.5cm) --++(-3,0) arc(-90:-180:.5cm) --++(0,2.5);

\node at (5.25,3.75) {$\cdots$};

\draw[Box] (7, 3) rectangle (10,4);
\node at (8.5,3.5) {$h_n$};
\node[marked, scale=.8, above left] at (7,4) {};
\draw[thickline] (7.75, 4) --++(0,1.5);
\draw[thick] (8.5, 4) arc(180:0: 1.3cm and 1cm) --++(0,-1) arc (0:-90: 1.2 cm and 1cm) -- (8.75, 2) arc(90:180:.5cm);
\draw[thickline] (9.25,4) arc(180:0: .6cm and .5cm) --++(0,-1) arc(0:-90:.5cm) --++(-3,0) arc(-90:-180:.5cm) --++(0,2.5);

\draw[Box] (1,.5) rectangle (9,1.5);
\node at (5,1) {$x$};
\node[marked,scale=.8, above left] at (1,1.5) {};
\draw[thickline] (1,1) --++ (-1,0);
\draw[thickline] (9,1) --++ (1,0);
\node[right] at (11, 2) {$\in Gr_k\mc{P}$.};
\end{tikzpicture}
\]
Thus
\begin{align*}
\left\| \hat{c}_k\circ\eta_k(x) -\hat{c}_k(x_n) \right\| \leq \sum_{uu_2^\circ u_1\in L} |\beta_x(uu_2^\circ u_1)| \|u_1\|\|u_2\| \|\hat{c}\circ\eta(u) - [\hat{c}(u)](\D_c\hat{c}(h_n)) \| \to 0,
\end{align*}
since $x \in Gr_k\mc{P}$ has finite support in $Gr_k\mc{P}^\Gamma$. Thus $M_k^{(v)}\subset M_k$.

The reverse inclusion follows from the same argument since we showed in the proof of Theorem \ref{vNas_are_equal} that $\hat{c}(u)$ is the $\|\cdot\|_R$-norm limit of elements of the form $\hat{c}\circ\eta(u')$.

The final statements are immediate from the equalities established above, but we also note that they follow from the fact that $I_k^{k-1}$ intertwines $\eta_k$ and $\eta_{k-1}$ for each $k$.
\end{proof}

\begin{rem}
As with Theorem \ref{vNas_are_equal}, Theorem \ref{k_vNas_are_equal} also holds when the von Neumann algebras are replaced with the corresponding $C^*$-algebras.
\end{rem}

One should think of the embeddings $\hat{c}_k\circ\eta_k$, $k\geq 0$ as small perturbations of the embeddings $\hat{c}_k$ of $Gr_k\mc{P}$. Thus, Theorems \ref{vNas_are_equal} and \ref{k_vNas_are_equal} say that when the perturbation is small enough, the von Neumann algebras generated by the $Gr_k^+\mc{P}$ are the same and we can recover the subfactor planar algebra $\mc{P}$ as the standard invariant of the subfactors $i^{k-1}_k(M_{k-1,+}^{(v)})\subset M_{k,+}^{(v)}$.

Suppose $\tau_0\colon Gr_0^+\mc{P}\to \C$ is a trace and let $f\in Gr_0^+[[\mc{P}]]$ be such that $\tau_0(x)=\<f^*, x\>$. Recall that we can extend this to a series of traces $\tau_k\colon Gr_k^+\mc{P}\to \C$, $k\geq 0$, via (\ref{duality}). Let $(\H_k,\pi_k,\xi_k)$ be the GNS representation of $(Gr_k^+\mc{P},\wedge_k)$ with respect to $\tau_k$, and let $L_k=\pi_k(Gr_k^+\mc{P})''\subset \B(\H_k)$. The inclusion tangles $I_k^{k-1}$ induce inclusions $\hat{i}_k^{k-1}\colon \pi_{k-1}(Gr_{k-1}^+\mc{P})\to \pi_k( Gr_k^+\mc{P})$ such that $\hat{i}_k^{k-1}\circ \pi_{k-1} = \pi_k \circ I_k^{k-1}$. Thus when the $L_k$ are factors, one can consider the standard invariant associated to these inclusions. The following corollary shows that if $f$ satisfies the Schwinger-Dyson planar tangle with a potential $v$ close enough to $v_0$, then $L_k\cong M_{k,+}$ for each $k\geq 0$ and hence the standard invariant for $\{L_{k}\subset L_{k+1}\}_{k\geq 0}$ is simply $\mc{P}$.

\begin{cor}
Let $\epsilon>0$ be as in Theorem \ref{vNas_are_equal} and $\{\tau_k\}_{k\geq 0}$ and $f\in Gr_0^+[[\mc{P}]]$ as above. Suppose $f$ satisfies the Schwinger-Dyson planar tangle with potential $v\in (Gr_0\mc{P})^{(R'+1,\sigma)}_{c.s.}$. If $\|v-v_0\|_{R,\sigma}<\epsilon$, then there exists trace-preserving embeddings $(Gr_k^+\mc{P},\tau_k)\hookrightarrow (\mf{M}_k,\varphi_k)$ for each $k$, and the von Neumann algebra generated by $Gr_k^+\mc{P}$ under this embedding is $M_k$. Moreover, $L_k\cong M_{k,+}$ for each $k\geq 0$.
\end{cor}
\begin{proof}
Let $g\in (Gr_0\mc{P})^{(R',\sigma)}_{c.s.}$ be the transport element from $v_0$ to $v$. Then the embeddings are simply $\{\frac{1}{|V_+|}\hat{c}_k\circ \eta_k\}_{k\geq 0}$ and the equality of the generated von Neumann algebras follows from Theorem \ref{k_vNas_are_equal}. The isomorphism $L_k\cong M_{k,+}$ follows from the fact that both representations $\pi_k$ and $\hat{c}_k\circ \eta_k$ are trace-preserving.
\end{proof}

\end{document}